\documentclass[12pt]{amsart}

\usepackage{amsthm, amsmath, amscd, amssymb, latexsym, stmaryrd, color}

\usepackage[all]{xypic}
\usepackage{color}
\usepackage[dvipsnames]{xcolor}
\usepackage{fullpage}
\definecolor{hot}{RGB}{65,105,225}
\usepackage[pagebackref=true,colorlinks=true, linkcolor=Blue,  citecolor=BlueViolet, urlcolor=hot]{hyperref}

\usepackage[%
    left=1.in,%
    right=1.in,%
    top=1.5in,%
    bottom=1.5in,%
    paperheight=11in,%
    paperwidth=8.5in%
]{geometry}
\usepackage[all,arc]{xy}
\usepackage{enumerate}

\usepackage{mathtools}
\usepackage{mathrsfs}
\usepackage{graphicx}
\usepackage{hyperref}
\usepackage{parskip}
\usepackage{fancyhdr}

\usepackage[T1]{fontenc}
\usepackage{newtx}
\usepackage{tikz}
\usepackage{tikz-cd}
\usepackage{dsfont}
\usetikzlibrary{matrix}

\newtheorem{thm}{Theorem}[section]
\newtheorem{cor}[thm]{Corollary}
\newtheorem{prop}[thm]{Proposition}
\newtheorem{lem}[thm]{Lemma}

\newtheorem{rem}[thm]{Remark}
\newtheorem{defn}[thm]{Definition}

\makeatletter
\let\c@equation\c@thm
\makeatother
\numberwithin{equation}{section}

\makeatletter
\def\thm@space@setup{%
  \thm@preskip=\parskip \thm@postskip=0pt
}
\makeatother

\title{Exact fillings using universal covers in Floer theory. }
\date{}
\author{Shahnaz Shamim Shahul}
\address{Université de Toulouse\\CNRS\\Institut de Mathématiques de Toulouse \\ France\\ }
\email{shahnaz\_shamim.shahul@math.univ-toulouse.fr}

\begin{document}
\maketitle

\begin{abstract}
   We study exact fillings of Legendrians and contact manifolds using Floer theory of universal covers. We prove  vanishing results on wrapped Floer homology of the universal cover and symplectic (co)homology with $k[\pi_1]$ coefficients by using lifted versions of Ritter's TQFT operations. We use these results to show that an exact Maslov zero Lagrangian filling of the standard Legendrian $\Lambda_{std}$ in any Liouville filling of $(S^{2n-1},\xi_{std})$ is diffeomorphic to a ball in dimension at least 5. We show that the existence of a topological simple exact filling $W$ with vanishing symplectic homology for a dynamically convex contact manifold $M^{2n-1}$ reveals information about homotopy groups of $W$ and $M.$ We also strengthen a theorem of Zhou for ADC-contact manifolds as an application of universal covers.
\end{abstract}

\section{Introduction}
It is a central question in symplectic and contact topology to study symplectic fillings of a  contact manifold or, in the relative setting, Lagrangian fillings of a  Legendrian. Floer theory has been used effectively to study obstructions to the existence of such fillings. For a non-exhaustive list of references see \cite{lazarev2020contact},\cite{oancea2012topology},\cite{zhou2021symplectic},\cite{Zhou_2024}in the contact setting and \cite{kim2024topology},\cite{chantrainefloer} in the Legendrian case. 

Covering spaces and local systems have also been used in Floer theory with great effectiveness. In \cite{damian2012floer}, Damian constructed Lagrangian Floer homology of a covering space and used it to prove constraints on the Maslov number of a monotone embedding. \cite{chantrainefloer} used relations with bilinearized Legendrian contact homology and a construction of wrapped Floer homology of covering spaces to prove that for the standard Legendrian $\Lambda_{std}:={\{x+iy\in S^{2n-1}\subset \mathbb{C}^n, y=0  \}}$, every exact Lagrangian filling in the standard symplectic filling of $(B^{2n},d\lambda_{std})$ is diffeomorphic to $B^n$. Zhou also studied obstructions to symplectic fillings using covering spaces in \cite{zhou2021symplectic} and \cite{zhou2023fillings}. Applications of local systems in Floer theory can be seen in \cite{albers2017local} and higher-local systems in \cite{barraud2007lagrangian},\cite{barraud2024floer}.

Note that the work of \cite{chantrainefloer} is a partial generalization of a result of Eliashberg-Polterovich \cite{eliashberg1996local}
  which states that there is a unique exact Lagrangian filling for the standard Legendrian $\Lambda_0$ in $(S^3,\xi_{std})$ with it's unique exact filling diffeomorphic to $B^2$ in the standard ball $(B^{4},d\lambda_{std})$.
In \cite{kim2024topology}, Kim and Kwon advanced this work by studying the topology of an exact Maslov-zero filling $L$ of the standard Legendrian $\Lambda_{std}$ in \textit{any Liouville filling} of the $S^{2n-1}$. Using wrapped Floer homology and index-positivity of $\Lambda_{std}$, they proved that $L$ must necessarily be a homology ball and in the more restrictive case of real Lagrangian filling, they proved it is diffeomorphic to a ball using  holomorphic foliations. In the paper, they remark that one can prove the more general case of $L$ being diffeomorphic to a ball in a similar fashion to \cite{chantrainefloer} if one can prove that the wrapped Floer homology of the universal covering of $L$ vanishes. We in fact follow their remark  by proving that the wrapped Floer homology of the universal cover vanishes. 

\begin{rem}
    Note that by the Eliashberg-Floer-Mcduff theorem (\cite{mcduff1991symplectic}), any exact filling of $(S^{2n-1},\xi_{std})$ is  diffeomorphic to a ball $B^{2n}$, however it might not be the standard symplectic ball ($B^{2n},d\lambda_{std}$).
\end{rem}

Let $\tilde{L}$ be a covering of a Lagrangian $L$ in a Liouville domain. We construct a lifted version of wrapped Floer homology denoted by $HW^{(0)}(\tilde{L})$ following the construction of Damian in the setting of closed Lagrangians (see \cite{damian2012floer}). We prove that the lifted wrapped Floer homology is a module over the wrapped Floer homology of the base $L$ by constructing lifted TQFT maps in wrapped Floer homology(Theorem \ref{module}). .

\begin{thm}[Theorem \ref{module}]\label{mod}

    The lifted wrapped Floer homology for a covering $\tilde{L}$ of $L$, $ HW^{(0)}(\tilde{L})$ forms a module over the the wrapped homology of the base $HW(L)$.

\end{thm}

As a consequence the vanishing of the the wrapped Floer homology of $L$ implies the vanishing  of the lifted wrapped Floer homology of any cover $\tilde{L}$ (Corollary \ref{Vanishing WH}). 

Like the (non-lifted) positive wrapped Floer homology, the lifted positive wrapped Floer homology can be proven to be independent of the Lagrangian filling under suitable assumptions(see proposition \ref{adccover}).

In the case of a Maslov-zero exact Lagrangian filling $L$ of the standard Legendrian $\Lambda_{std}$ in $(S^{2n-1},\xi_{std})$, we can obtain information about $\pi_1(L)$  (Theorem \ref{funda1}). We thereby deduce the following theorem.

\begin{thm}[Corollary \ref{lagrangian filling proof}]\label{lagrangian filling}
    Let $W^{2n}$ be a Liouville filling of $(S^{2n-1},\xi_{std})$ and let $L$ be an exact Maslov-zero Lagrangian filling of the standard Legendrian $\Lambda_{std}$ . Then $L$ is diffeomorphic to the ball $B^n$ for $n\neq 4$ and homeomorphic to $B^n$ when $n=4$ .
\end{thm}
\begin{rem}
    We have been informed by Zhengyi Zhou that he and  Lite Du have independently proved theorem \ref{lagrangian filling} without the Maslov zero assumption, by a different method and are currently revising their draft (\cite{zhoudulagrangian}) for publication.
\end{rem}

\begin{rem}
    In fact the idea of the proof on proving $\pi_1$ is trivial (Theorem \ref{funda1}) using vanishing of lifted wrapped Floer homology can also be applied to Maslov-zero exact fillings of Legendrians that are not necessarily the standard Legendrian but have some restrictions on Conley-Zehnder indices. A particular example is given in Theorem \ref{funda}.
\end{rem}

In the closed setting, we also prove vanishing results for symplectic homology with $k[\pi_1]$ coefficients (corollory \ref{loccoeff}) as well as it's dual version in symplectic cohomology with $k[\pi_1]$ coefficients (theorem \ref{vansc}). We use these results to study topological restrictions on dynamical convexity and on fillings of ADC-contact manifolds. 

A contact form  $\alpha$ in a contact manifold $M$( with $c_1(M)$ torsion ) is said to be \textbf{dynamically convex} if the Conley-Zehnder index of all contractible non-degenerate Reeb orbits are bounded below by $n+1$. A contact manifold is dynamically convex if it admits a dynamically convex contact form. The condition of dynamical convexity was introduced in \cite{hofer1999characterization} and has been of importance in symplectic dynamics( for instance see \cite{ginzburg2021dynamical},\cite{chaidez20223d},\cite{dardennes2024symplectic},\cite{abreu2017multiplicity}). In \cite{shahul2026topologicalobstructionsdynamicalconvexity},we motivated studying the topology of dynamical convexity from the viewpoint of algebraic geometry using the results of Mclean\cite{mclean2016reeb} and a conjecture of Shokurov in algebraic geometry. 

Kwon-Oba \cite{kwon2024dynamically} proved that Weinstein fillings of dynamically convex contact manifolds  having vanishing symplectic homology are homology balls, and using h-cobordism when $\pi_1=0$,are  diffeomorphic to balls. We prove that we can obtain homotopical information of  fillings without the condition on  $\pi_1$ or being Weinstein  by using both symplectic homology and symplectic cohomology  of covering spaces. More precisely, we prove the following statement.

\begin{thm}[Theorem \ref{hom},\ref{homotopy}]\label{homotop}
    Let $M^{2n-1}$ be a dynamically convex contact manifold that admits an exact filling $W$ with vanishing symplectic homology, then $\pi_i(W)=0$ for $i>1$. If $W$ has a finite fundamental group, then $\pi_k(M)=0$ for $1<k<2n-1$.
\end{thm}

Note that all examples of dynamically convex manifolds known to the author are either hypertight contact manifolds or quotients of the standard sphere $S^{2n-1}$. In the latter case, one has $\pi_i(M)=0$ for $1<i<2n-1$ and $\pi_1$ is finite . For hypertight contact manifolds, note that $\pi_1(M)$ is necessarily infinite since otherwise one could concatenate the orbits to have a contractible loop. 

Another application of using both symplectic homology and symplectic cohomology of a filling $W$ is that we can gather information on Reeb orbits from the topology of $M$ without the more restrictive assumption of dynamical convexity.

\begin{thm}[Theorem \ref{cz}]\label{czi}
    Let $M^{2n-1}$ be a contact manifold such that it admits an exact filling $W$ with $SH_*=0$. If there exists a non-degenerate contact form $\alpha$ on $M$ with no Reeb orbits of Conley Zehnder index $2+k$ and $1-k$, then $H_{k+n}(\partial W)=0$ and $H_{n-k-1}(\partial W,\mathbb{Z})=0$.
\end{thm}

We show that in addition to working with symplectic homology of universal covers, symplectic cohomology can give us additional information by strengthening a result of Zhou where he uses symplectic homology of the cover. In \cite{zhou2021symplectic} Theorem E, Zhou proves that given a topologically simple exact filling $W$ of an ADC-contact manifold  $M$  with vanishing symplectic homology, if the map $\pi_1(M)\rightarrow \pi_1(W)$ induced by inclusion is an isomorphism, then for every topologically simple exact filling $W'$, $\pi_1(M)\rightarrow \pi_1(W')$ is an isomorphism. We can generalize the above result to $\pi_2$ as a cute application of working with the symplectic cohomology of the universal cover.

\begin{defn}\label{adc defn}
    A contact manifold $(M^{2n},\xi)$ with $c_1(M,\mathbb{Q})=0$ is asymptotically dynamically convex if there exists a sequence of contact forms $\alpha_1>\alpha_2>...$ and increasing positive numbers $D_i$ that tends to infinity such that $\mu_{CZ}(\gamma)>3-n$ for any $\gamma $, a non-degenerate Reeb orbits of $\alpha_i$ with action bounded above by $D_i$
\end{defn}
\begin{rem}
    ADC(Asymptotically dynamically convex) (see definition \ref{adc defn}) contact manifolds were introduced by Lazarev(\cite{lazarev2020contact}) that generalize index positive contact manifolds where one can derive independence of Floer theoretic invariants like positive symplectic homology from its filling. The class of manifolds is rich enough for many applications (for instance see \cite{bowden2022tight},\cite{zhou2021symplectic},\cite{Zhou_2024})
\end{rem}

\begin{thm}[Theorem \ref{zh}]\label{zhou}
    Let $M$ be an ADC contact manifold with a topologically simple exact filling $W$ such that $SH_*(W)=0$ and $\pi_i(M)\rightarrow \pi_i(W)$ is an isomorphism for $i=1,2$. Then for any other topologically simple exact filling $W'$, $\pi_i(M)\rightarrow \pi_i(W')$ is an isomorphism for $i=1,2$.
\end{thm}

For higher $\pi_i$, we have a weaker version of the above theorem.
\begin{thm}[Theorem \ref{higherhomproof}]\label{higherhom}
    Let $M$ be an ADC contact manifold with a topologically simple exact filling $W$ such that $SH_*(W)=0$ and $\pi_i(M)\rightarrow \pi_i(W)$ is an isomorphism for $i\leq k$. Then for any other topologically simple exact filling $W'$, $\pi_i(M)\rightarrow \pi_i(W')$ is an isomorphism for $i\leq k-2$.
\end{thm}

Note that the above theorems gives obstructions to Weinstein fillability. Given a Weinstein filling $W^{2n}$ of $M$, the map $\pi_i(M)\rightarrow \pi_i(W)$ is an isomorphism for $i\leq n-1$ and surjective  at $i=n$, since there are no handles of index greater than $n$. Note that Weinstein fillings are topologically simple if $n\geq 3$. Therefore if $M$ admits a topologically simple filling $V$ with vanishing $SH$ such that $\pi_i(M)\rightarrow\pi_i(V)$ is not an isomorphism for $i\leq n-3$, then it cannot admit Weinstein fillings of vanishing $SH$. For instance, it cannot admit flexible Weinstein fillings in this case (Note flexibility implies vanishing $SH$ and ADC).

    \begin{rem}
    One could ask if the vanishing results of lifted Floer theory in theorem \ref{Vanishing WH} and theorem \ref{loccoeff} still hold for arbitrary local coefficients (on the manifold $M$ in the closed case or the Lagrangian $L$ in the open case). We can use the above theorems to prove this is the case with local as well as DG local coefficients on $M$ or $L$, which will appear as part of future work.
\end{rem}

\subsection{Organization of the paper} In section \ref{2}, we review symplectic homology and its dual symplectic cohomology followed by proving the vanishing results of symplectic homology (theorem  and symplectic cohomology with $k[\pi_1]$ coefficients  in section \ref{2.1} and section \ref{2.2} respectively). In section \ref{3}, we review wrapped Floer homology followed by construction of lifted wrapped Floer homology and lifted TQFT maps to prove a vanishing result in this case (corollary \ref{Vanishing WH}). In section \ref{4}, we use the previous the results for applications. We prove theorem \ref{homotop} on dynamical convexity in section \ref{4.1}, we prove theorem \ref{czi} on Conley-Zehnder indices on section \ref{4.2}, theorem \ref{lagrangian filling} on Lagrangian fillings in section \ref{4.3} and  theorem\ref{zhou} on ADC contact manifolds in section \ref{4.4}

\subsection{Acknowledgments}
I would like to thank my advisors Jean François Barraud and Jean Gutt for their advice and constant support. I would also like to thank Zhengyi Zhou for his kind and valuable comments as well as pointing out an issue on a early draft. I would  like to acknowledge the financial support provided by EUR-MINT(ANR-18-EURE-0023) at the Institut de Mathématiques de Toulouse.

\subsection{AI declaration} The author would like to state that no generative AI tools were used in this work.

\section{Symplectic Homology and Cohomology }\label{2}
In this section we recall some basics about symplectic homology and cohomology of a Liouville domain $W$ which was first introduced by Viterbo(\cite{viterbo1999functors}) . We then proceed by symplectic homology and cohomology of the covering space $\tilde{W}$. Afterwards we show various vanishing results of symplectic (co)homology twisted by $\mathbb{Z}[\pi_1]$ coefficients leveraging the use of product structures. We follow grading and sign  conventions  as in \cite{lazarev2020contact},\cite{cieliebak2010rabinowitz} which we also recommend as a more detailed introduction to symplectic homology.

Let $\hat{W}:=W\bigsqcup_{\partial W} [1,\infty)\times \partial W $ be the Liouville completion of a Liouville domain $W$ obtained by attaching the positive part of the symplectization $[1,\infty)\times \partial W$ to $W$ at $\partial W$. We extend the Liouville form $\lambda $ on $W$ to the completion $\hat{W}$ by $r\lambda|_{\partial W}$ on $[1,\infty)\times \partial W$ where $r$ denotes the cylindrical coordinate.

A Hamiltonian $H$ on $\hat{W}$ is is called \textbf{admissible} if $H$ is $C^2-$small in $W$ and $H=sr+c$ in $\hat{W}\textbackslash W$, where $s$ denotes the slope and $c\in \mathbb{R}$. The set of admissible Hamiltonians is denoted by $H_{adm}$.
We call a  class $J_{adm}$ on $\hat{W}$ of  almost complex structures admissible if for $J\in J_{adm}$, we have $J\ker\alpha=\ker\alpha$, $J|\ker\alpha$  is a compatible almost complex structure on the bundle $\ker \alpha$, and $J$ is independent of $r$, $J(r\partial_r)=R_\alpha$ on $\hat{W} \textbackslash W $ and $J$ is compatible with the symplectic form on $\hat{W}$.

For our choice of $H\in H_{adm}$, the periodic trajectories $\mathcal{P}(H)$of the Hamiltonian vector field $X_H$ correspond to Morse critical points on $W$ and Reeb orbits on a level $r$ in $\hat{W}/W$ of period $s$ since $X_H=sR_\alpha.$ The Hamiltonian orbits are degenerate in the cylindrical end and we  consider a small time-dependent perturbation of $H$, so that each $S^1$ family degenerates into two orbits\cite{bourgeois2009exact}.

Given $H\in H_{adm}$ and $J\in J_{adm}$, we define $SC(W,H,J):=\bigoplus _{\gamma\in \mathcal{P}(H)}{\mathbb{Z}}\gamma$ where the Hamiltonian is considered after a small-time dependent perturbation. When $c_1(W,\mathbb{Q})=0,$ we can associate a rational grading by $\mu^\mathbb{Q}_{CZ}$ to contractible orbits in $\mathcal{P}(H)$ and therefore to $SC(W,H,J)$ (cf\cite{mclean2016reeb}).

After a small-time dependent perturbation of $(H,J)$, the moduli space of Floer trajectories connecting two orbits  $x_-,x_+$,denoted by $$\mathcal{M}(x_+,x_-,H,J):=\{u:\mathbb{R}\times S^1\rightarrow \hat{W};\lim_{s\rightarrow\pm\infty}u(s,.)=x_\pm,\partial_su+J\partial_tu-JX_H=0\}/\mathbb{R}$$ is a manifold of dimension $\mu_{CZ}(x_-)-\mu_{CZ}(x_+)-1$. Using a maximum principle, it can be shown that these Floer trajectories do not escape to infinity and $\mathcal{M}(x_+,x_-,H,J)$ can be compactified using broken trajectories.
$\mathcal{M}(x_+,x_-,H,J)/\mathbb{R}$  is therefore  a compact  $0-$dimensional manifold when $\mu_{CZ}(x_+)-\mu_{CZ}(x_-)-1=0$.

We  define a differential $d:SC_{k}\rightarrow SC_{k-1}$  such that $$d x_+=\sum_{x_-\in SC_{k-1}}|\mathcal{M}(x_+,x_-,H,J)|x_-.$$   After considering $1-$dimensional moduli spaces and their breakings one can show that $d\circ d=0$ and obtain the homology group of the chain complex which we denote by $SH_*(W,H,J)$.

Note that we define the differential $d:SC_{k}\rightarrow SC_{k-1}$ by $d x_+=\sum_{x_-\in SC_{k-1}}|\mathcal{M}(x_+,x_-,H,J)|x_-$ where we look at the rigid trajectories that flow into an orbit $x_+$. We can also define a differential $\partial$ by counting trajectories that flow out of $x_+.$ More explicitly we can define $\partial:SC^{k-1}\rightarrow SC^{k}$ by $d x_-=\bigoplus_{x_-\in SC_{k-1}}|\mathcal{M}(x_+,x_-,H,J)|x_+$.(Here $SC^k$ consists of orbits of index $k$ just like $SC_k$ but we write with the cohomological notation since the differential $\partial$ increases grading).

The same analysis of the moduli spaces in $SH_*(W,H,J)$ implies that $\partial\circ\partial=0$, and we denote homology of the cochain complex $SC^*$ by $SC^*(W,H,J)$

When slope of $H_1$ is less than the slope of $H_2$, there exists a continuation map $SH_*(W,H_1,J_1)\rightarrow SH_*(W,H_2,J_2)$ and these maps form a directed system. Dually we get an inverse system $SH^*(W,H_2,J_1)\rightarrow SH^*(W,H_1,J_2)$ in cohomology.
We define the \textbf{symplectic homology} $SH_*(W)$ of W to be direct limit $\underset{slope(H)}{\varinjlim} SH_*(W,H,J)$ and \textbf{symplectic cohomology} $SH^*(W)$ to be the inverse limit $\underset{slope(H)}{\varprojlim} SH^*(W,H,J)$ .

There exists a canonical filtration on symplectic homology and cohomology induced by the action functional $\mathcal{A}_H$ for a small-time dependent perturbation of $H\in H_{adm}.$ Given $\gamma \in\mathcal{L}W$,  the action functional is defined by $\mathcal{A}_H(\gamma)=\int_{S^1}\gamma^*\lambda-\int_{S^1}H(\gamma(t))dt$. 

Note that, following the convention  in \cite{lazarev2020contact}, for symplectic homology we consider the positive gradient flow in Floer trajectories which increases the action and therefore the differential decreases the action. This  implies the existence of  a well-defined filtered subcomplex $SC_*^{<k}(W,H,J)$ of $SC_*(W,H,J)$ whose homology is denoted by  $SH_*^{<k}(W,H,J)$. Taking direct limits we can define the filtered symplectic homology groups $SH_*(W)^{<k}:=\underset{slope(H)}{\varinjlim} SH_*(W,H,J)^{<k}$. 

For a small time dependent perturbation of $H$ and $\epsilon>0$ small enough, we have that all the Hamiltonian orbits with action less than $\epsilon$ of $\hat{W}$ correspond to the Morse critical points in $W$ as $H$ is $C^2-$small. That is $SC^{<\epsilon}(W,H,J)$ is the Morse complex with a grading shift. The subcomplex $SC_*^{+,<\epsilon}:=SC_*/SC_*^{<\epsilon}$  is generated by Hamiltonian orbits corresponding to the Reeb orbits in the symplectization. We obtain an invariant  by taking the direct limit of the homology $SH_*^{+,<\epsilon}$ of this chain complex which is called \textbf{positive symplectic homology} $SH_*^+$.
\begin{equation}   
SH_*^+(W)=\underset{slope(H)}{\varinjlim} SH^+_*(W,H,J)^{<k}
\end{equation}

By construction the short exact sequence of chain complexes $0\rightarrow SC^{<\epsilon}\rightarrow SC\rightarrow SC^+\rightarrow 0$ gives an exact triangle between the Morse homology, symplectic homology and the positive symplectic homology.

\begin{prop}\label{les}
    For a Liouville domain $W^{2n}$, there exists a long exact sequence of groups
   
\[\begin{tikzcd}[cramped]
	{...} & {H_{*+n}(W,\partial W)} & {SH_{*}(W)} & {SH_{*}^{+}(W)} & {H_{*+n-1}(W,\partial W)} & {...}
	\arrow[from=1-1, to=1-2]
	\arrow[from=1-2, to=1-3]
	\arrow[from=1-3, to=1-4]
	\arrow[from=1-4, to=1-5]
	\arrow[from=1-5, to=1-6]
\end{tikzcd}\]

\end{prop}

In symplectic cohomology, differential $\partial$ increases the action. Therefore the filtered complex associated to $SC^*$ is given by $SC^*_{>k}$ where the subcomplex is generated by orbits bounded below by action $k$. Taking inverse limits, we define filtered symplectic cohomology groups $SH^*_{>k}(W)=\underset{slope(H)}{\varprojlim} SH^*_{>k}(W,H,J)$. Here for a small $\epsilon$, the subcomplex $SC^*_+:=SC^*_{>\epsilon}$ is generated by Hamiltonian orbits corresponding to Reeb orbits and the quotient complex $SC^*/SC^*_+$ is the morse complex of $W$ generated by morse critical points. We define the \textbf{positive symplectic cohomology group} to be the inverse limit of $SH^*_+:=SH^*_{>\epsilon}$

\begin{equation}   
SH^*_+(W)=\underset{slope(H)}{\varprojlim} SH^*_+(W,H,J)^{+,>\epsilon}
\end{equation}

Therefore by construction we have the the following dual exact triangle.

\begin{prop}\label{les2}
For a Liouville domain $W^{2n}$, there  exists a long exact sequence such that
\[\begin{tikzcd}
	{...} & {H_{*+n+1}(W)} & {SH^{-*}_{+}(W)} & {SH^{-*}(W)} & {H_{*+n}(W)} & {...}
	\arrow[from=1-1, to=1-2]
	\arrow[from=1-2, to=1-3]
	\arrow[from=1-3, to=1-4]
	\arrow[from=1-4, to=1-5]
	\arrow[from=1-5, to=1-6]
\end{tikzcd}\]
\end{prop}

The following theorem relates symplectic cohomology and symplectic homology over field coefficients $k$. In particular it implies that vanishing of symplectic cohomology is equivalent to the vanishing of symplectic homology over field coefficients.

\begin{thm}[\cite{cieliebak2018symplectic}\cite{ritter2013topological}]\label{adual}(Algebraic Duality) Let $W$ be a Liouville domain and let $k$ be a field. Then $SH_*(W,k)\cong Hom(SH^*(W,k),k)$
    
\end{thm}

\begin{rem}
    Symplectic homology $SH_*(W)$ and symplectic cohomology $SH^*(W)$ can also be defined using a Hamiltonian $H$ such that $H$ is quadratic (for instance $r^2$) in the collar $[1,\infty]\times\partial W $ and $C^2$-small in $W$. For such a quadratic Hamiltonian denoted by $H_q$, the homology of the chain complex $SC_*(W,H_q)$ is isomorphic to $SH_*(W)$ defined as the direct limit of linear Hamiltonians. However in order to have a well-defined cochain complex $SC^*(W,H_q)$, it is necessary to use a direct product instead of a direct sum since the differentials are action increasing and the the generators could be infinite. This is not a problem with the definition using linear Hamiltonian because  the differential is a finite sum for action reasons. For more details and a proof of the isomorphism, see \cite{ritter2013topological}.
\end{rem}

\subsection{Symplectic homology of covering spaces.} \label{2.1}

Zhou \cite[Section 3.3]{zhou2021symplectic}\cite{zhou2023fillings}  defines and uses symplectic homology of covering spaces to study fillings of asymptotically dynamically convex contact manifolds. Given a covering $\Tilde{W}\xrightarrow{\pi}W$, $SH(\Tilde{W},H)$ is defined as the group generated by lifts of the periodic orbits of $W$ and the differential is determined by counting the moduli spaces of (unique) lifted trajectories. More precisely the chain group is defined as $SC_*(\Tilde{W},H):=\bigoplus_{x\in\mathcal{P}(H)} \Pi_{\pi(\Tilde{x})=x}\mathbb{Z}\Tilde{x}$ where $\tilde{x}$ is a lift of the contractible Reeb orbit $x$ and it could be an infinite product in the fiber direction.
The differential $\tilde{d}:SC_*(\Tilde{W},H)\rightarrow SC_{*-1}(\Tilde{W},H)$ is given by $d\tilde{x}=\bigoplus_{{y}}\tilde{\mathcal{M}}(\tilde{x},\tilde{y})x$ where $\tilde{\mathcal{M}}(\tilde{x},\tilde{y})$ is the moduli space of rigid trajectories in $\mathcal{M}(x,y)$ that can be lifted to trajectories from $\tilde{y}$ to $\tilde{x}$ or equivalently the count of lifted trajectories. The differential $\tilde{d}$ is well defined since the moduli space $\mathcal{M}(x,y)$ is compact and using unique lifting property there are only finitely many $\tilde{y}$ after $\tilde{x}$ is fixed.

The canonical map $x\rightarrow \Pi_{\pi(\Tilde{x})=x}\Tilde{x} $ induces maps on the level of $SH^{<\epsilon},SH, SH^+$ and fits into the long exact sequence of symplectic homology. Furthermore the map is a unital ring map in $SH$ \cite[4.1]{zhou2023fillings}.

  \[\begin{tikzcd}
	{...} & {H^*(W)} & {SH_{*-n}(W)} & {SH^+_{*-n}(W)} & {...} \\
	{...} & {H^*(\tilde{W})} & {SH_{*-n}(\tilde{W})} & {SH^+_{*-n}(\tilde{W})} & {...}
	\arrow[from=1-1, to=1-2]
	\arrow[from=1-2, to=1-3]
	\arrow[from=1-2, to=2-2]
	\arrow[from=1-3, to=1-4]
	\arrow[from=1-3, to=2-3]
	\arrow[from=1-4, to=1-5]
	\arrow[from=1-4, to=2-4]
	\arrow[from=2-1, to=2-2]
	\arrow[from=2-2, to=2-3]
	\arrow[from=2-3, to=2-4]
	\arrow[from=2-4, to=2-5]
\end{tikzcd}\]

\begin{rem}
Zhou\cite[Section 3.3]{zhou2021symplectic}\cite{zhou2023fillings} uses cohomological convention and defines symplectic cohomology groups. They coincide with our definition if we define our grading using $-\mu_{CZ}$ instead of $\mu_{CZ}$. He also works with quadratic Hamiltonians. If we want to work  with linear Hamiltonians, it suffices to take the direct limit of $\varinjlim {SH(\Tilde{W},H)}$ induced by lifts of moduli spaces involved in the definition of continuation maps.
   \end{rem}

Zhou \cite[Remark 4.1]{zhou2023fillings} also discusses symplectic homology twisted by $\mathbb{Z}[\pi_1]$ coefficients and makes a comparison with the symplectic homology of covering spaces. In this case, the chain group is the $\mathbb{Z}[\pi_1]$ module generated by the periodic orbits with a twisted differential by the holonomy action of the fundamental group. When $\pi_1$ is finite, symplectic homology of the universal covering space coincides with the symplectic homology twisted with $\mathbb{Z}[\pi_1]$ coefficients since the covering space is compact and there are finitely many lifts of the generators. When $\pi_1$ is infinite, the groups are different since symplectic homology of the universal covering space has infinitely many generators in a fiber direction.

\begin{rem}
    Singular cohomology of $W$ with $\mathbb{Z}[\pi_1]$ local coefficients is isomorphic to the compactly supported cohomology of the universal covering space $H^c(\tilde{W})$ and not $H(\tilde{W})$. However the cup product makes $H^c(\tilde{W})$, a module over  $H(\tilde{W})$.    
\end{rem}

\begin{thm}\label{modulesc}
    Given a Liouville domain $W$, the pair of pants product makes the symplectic homology of $W$, twisted by $\mathbb{Z}[\pi_1]$ local coefficients a module over $SH_*(\tilde{W})$, the symplectic homology of the universal covering $\tilde{W}$.
\end{thm}
\begin{proof}
    We notice that the chain complex of symplectic homology with $\mathbb{Z}[\pi_1]$ local coefficients is a  subcomplex of that of symplectic homology of the universal cover $SH_*(\tilde{W})$ since its differential coincides with that of $SH_*(W,\mathbb{Z}[\pi_1])$ when restricted to finitely many orbits. Now let $\tilde{x}*\tilde{y}$ denote the pair of pants product with an element $\tilde{y}$ represented by an infinite sum of orbits in the fiber of $y$ and a single orbit $\tilde{x}$ in the fiber of $x$ . 

    Let $\mathcal{M}(x,y;z)$ denote the rigid moduli space of pair of pants with incoming ends at $x,y$ and outgoing end at $z$.The unique lifting property holds for a pair of pants $P$ in $\mathcal{M}(x,y;z)$ as $P$ deform retracts to a figure $8$ whose fundamental group is generated by any two boundary components say $x$ and $y$. The image  of the fundamental group of pair of pants by the map $P$, $Im (\pi_1(P))$ is trivial in $\pi_1(W)$ since $x$ and $y$ are contractible. Therefore the lifting lemma implies there is a unique lift of $P$ such that $\tilde{x}$ is a lift of $x$

    Let $z$ be an orbit in $W$. The rigid moduli space of pair of pants $\mathcal{M}(x,y;z)$ in $W$ is finite by compactness. Therefore after fixing $\tilde{x}$, there exists only finite lifts of  $\mathcal{M}(x,y;z)$ in $\tilde{W}$. Since the sum of rigid moduli spaces $\Sigma_z\mathcal{M}(x,y;z)$ is finite in $W$, there exists only finite lifts of $\sqcup_z\mathcal{M}(x,y;z)$ having $\tilde{x}$ and $\tilde{y}$ as entries. As a consequence the sum defining $\tilde{x}*\tilde{y}$ is finite implying $\tilde{x}*\tilde{y} \in SH_*({W},\mathbb{Z}[\pi_1(W)])$
     
     Therefore the product $z\in SC(W,\mathbb{Z}[\pi_1])$ and we have a module action $SH_*(\tilde{W})\otimes SH_*(W,\mathbb{Z}[\pi_1])\rightarrow SH_{*+*-n}(W,\mathbb{Z}[\pi_1])$.
\end{proof}

\begin{cor}\label{loccoeff}
    Given a Liouville domain $W$, $SH_*(W)=0$ implies $SH_*(W,{\mathbb Z[\pi_1]})=0.$
\end{cor}
\begin{proof}
    Let $\tilde{W}$ be the universal covering of $W$. Since the canonical map $SH_*(W)\rightarrow SH_*(\tilde{W})$ is a unital ring map, $SH_*(W)=0$ implies $SH_*(\tilde{W})=0$. Theorem \ref{module} implies that $SH_*(W,\mathbb{Z}[\pi_1])=0$ since the unit of $SH_*(\tilde{W})$ acts as the identity. 
    \end{proof}

\subsection{Symplectic cohomology with local coefficients} \label{2.2} We define symplectic cohomology with local coefficients following an analogous construction to the dual case of symplectic homology and prove a similar vanishing result. However we work over local coefficients $k[\pi]$  where  $k$ denotes a field since it is necessary for the existence of the coproduct in the level of the homology.

    Let $(SC^*(W),\delta,H)$ be the symplectic cochain complex associated to a Liouville domain $W$ and admissible Hamiltonian $H$.  We define $SH^*(W,k[\pi_1(W)])$  to be the homology of the twisted chain complex $(SC^*\otimes k[\pi_1(W)],H,\bar{\delta})$ with twisted differential $$\delta(x\otimes a)=\sum\limits_{CZ(y)=CZ(x)+1}y\otimes \bar{\mathcal{M}}(y,x).a$$ where rigid elements of $\bar{\mathcal{M}}(y,x)$ as seen as elements of the fundamental groupoid from $x$ to $y$ which act on $k[\pi_1(W)]$. The symplectic cohomology with $k[\pi_1(W)]$ coefficients is given by taking the inverse limit with respect to the continuation maps over admissible Hamiltonian $\mathcal{H}$,
    \begin{equation}
        SH^*(W,k[\pi_1(W)]):=\lim_{\substack{\longleftarrow\\H\in{\mathcal{H}}}} SH^*(W,H,k[\pi_1(W)])
    \end{equation}

We obtain by definition the following long exact sequence by filtering the twisted complex by action like in $\ref{les2}$ and identifying the homology with local coefficients ${H_{*}(\tilde{W},k)}$ with the homology of the universal cover $\tilde{W}$.

\begin{prop}\label{lesloc2}
For a Liouville domain $W^{2n}$ with universal cover $\tilde{W}$, there  exists a long exact sequence such that 
\[\begin{tikzcd}
	{...} & {H_{*+n+1}(\tilde{W},k)} & {SH^{-*}_{+}(W,k[\pi_1])} & {SH^{-*}(W,k[\pi_1])} & {H_{*+n}(\tilde{W},k)} & {...}
	\arrow[from=1-1, to=1-2]
	\arrow[from=1-2, to=1-3]
	\arrow[from=1-3, to=1-4]
	\arrow[from=1-4, to=1-5]
	\arrow[from=1-5, to=1-6]
\end{tikzcd}\]
    
\end{prop}

Let $x\otimes a$ , $y \otimes b$ and $z\otimes c$ be elements of $(SC^*\otimes k[\pi_1(W)])$. Since we are working exclusively with contractible orbits, these elements correspond to unique lifts $\tilde{x}$ ,$\tilde{y}$ and $\tilde{z}$ of $x$ ,$y$ and $z$ in the universal covering space $\tilde{W}$ with covering map $p:\tilde{W}\rightarrow W$. 

Symplectic cohomology with $k[\pi_1]$ local coefficients correspond to the homology of the cochain complex $(SC^*(\tilde{W},H,k):=\bigoplus_ {p(\tilde{x})=x}k\tilde{x},\tilde{\delta} ) $ where the the differential $\tilde{\delta}$ is defined by the count of unique lifts of elements of rigid moduli spaces followed by an inverse limit. This is because the twisted differential $\delta$ and the lifted differential $\tilde{\delta}$ coincide. Note after fixing a lift of the starting point, that the endpoints after lifting are in $1-1$ correspondence with the homotopy class of paths in $W$.

We now examine coproduct and counit on symplectic cohomology with $k[\pi_1]$ coefficients. For more details regarding their construction in the base case $SH^*(W)$, see \cite{ritter2013topological} with the exception that his grading convention makes it a homological grading.

The rigid moduli space of pair of pants  $\mathcal{M}(x,y;z)$ involved in the definition of the product in symplectic homology can be used to define a coproduct $\mu$ in symplectic cohomology given by $\mu(z)=\oplus_{x,y}\mathcal{M}(x,y;z)x\otimes y$. As in the case of the product, we can lift the images of pair of pants with contractible ends uniquely. This  construction of  lifted product in symplectic homology also gives  the existence of a \textit{lifted coproduct} $ \tilde{\mu}$ by exchanging the incoming for outgoing ends  in defining the map $$\tilde{\mu}:SH^{*+*-n}(W,k[\pi_1(W)])\rightarrow SH^*(W,k[\pi_1(W)])\otimes SH^*(W,k[\pi_1(W)])$$ on symplectic cohomology with $k[\pi_1]$ local coefficients. 

Dual to the unit in symplectic homology $SH_*$, we have a counit in symplectic cohomology denoted by $e:SC^*(W)\rightarrow k$. The counit $e$ is defined by counting the rigid moduli spaces $\mathcal{M}(x)$ of trajectories from $x$ parametrized by Floer equation near $x$ that can be capped off at $\infty$ by a $J-$holomorphic parametrization. The counit $e:SC^*(W,k)\rightarrow k$ is such that $e(x)=|\mathcal{M}(x)|$. We define the \textit{lifted counit} for $SH^*(W,k[\pi_1(W)])$ by the map $$e(\tilde x)=|\tilde{\mathcal{M}}(x)|=|\mathcal{M}(x)|.$$ Note that  by unique lifting the count of moduli space of lifts starting at $\tilde{x}$ is equal to the count of $\mathcal{M}(x)$.

The projection map $p:SC^*(W,k[\pi_1(W)])\rightarrow SC^*(W)$ is a cochain morphism and induces the following commutative diagram with $p_*$ defining a counital coproduct preserving map .
\[\begin{tikzcd}
	{...} & {SH^{-*}_+(W,k[\pi_1(W)])} & {SH^{-*}(W,k[\pi_1(W)])} & {H_{*+n}(\tilde{W})} & {...} \\
	{...} & {SH^{-*}_+(W)} & {SH^{-*}(W)} & {H_{*+n}(W)} & {...}
	\arrow[from=1-1, to=1-2]
	\arrow[from=1-2, to=1-3]
	\arrow["{p_*}"', from=1-2, to=2-2]
	\arrow[from=1-3, to=1-4]
	\arrow["{p_*}"', from=1-3, to=2-3]
	\arrow[from=1-4, to=1-5]
	\arrow["{p_*}"{description}, from=1-4, to=2-4]
	\arrow[from=2-1, to=2-2]
	\arrow[from=2-2, to=2-3]
	\arrow[from=2-3, to=2-4]
	\arrow[from=2-4, to=2-5]
\end{tikzcd}\]

\begin{thm}\label{vansc}
    Let $W$ be a Liouville domain such that $SH^*(W)=0$. Then $SH^*(W,k[\pi_1(W)])=0$
\end{thm}
\begin{proof}
    The symplectic cohomology $SH^*(W)=0$ implies that the counit $e$ satisfies $e=0$. Since the induced projection $p_*:SH(W,k[\pi_1(W)])\rightarrow SH(W)$ is counital, the lifted counit also satisfies $\tilde{e}=0.$ Therefore  $SH^*(W,k[\pi_1(W)])=0$
\end{proof}

Before concluding the section we would like to state a version of algebraic duality for symplectic homology of covering spaces in the flavor of theorem \ref{adual}.

Symplectic homology with $k[\pi_1]$ local coefficients  correspond to the direct limit of homology of the lifted complex $(SC_*(\tilde{W},H):=\bigoplus_ {p(\tilde{x})=x}k\tilde{x},\tilde{d} )$ where $\tilde{d}$ is defined by taking  unique lifts of the rigid moduli spaces involved in the differential $d$ in symplectic homology. Note that as mentioned in the proof of theorem \ref{modulesc}, this complex differs from the symplectic homology of the covering by not considering infinitely many generators in the fiber direction and is the compactly supported version of symplectic homology of covering spaces.

For a specific Hamiltonian $H$, similar to the differentials in standard symplectic homology and cohomology, $\tilde{d}$ and $\tilde{\delta}$ are defined using the same generators but differ depending on whether the trajectories flow into or out of the orbits. Therefore we adapt the argument in the   proof of \ref{adual} in \cite{ritter2013topological}(Section $3.8$) to prove a duality theorem with local coefficients.

\begin{prop}\label{adualc}
   Given a Liouville domain $W$, $SH_*(\tilde{W},H)\cong Hom(SH^*(W, k[\pi_1]),H,k).$
\end{prop}
\begin{proof}
    We can rewrite $SC^*(W,H,k[\pi_1(W)]) $ as $\bigoplus_{x\in \mathcal{P}(H)}\bigoplus_{p(\tilde{x}=x)}k$ where $p:\tilde{W}\rightarrow W$ is the covering projection,by using the identification with the lifted complex of the universal cover. Now applying the $Hom$ functor to this group, we get $Hom(SC^*(W,H),k)\cong \underset{x\in\mathcal{P}(H)}{\bigoplus}\underset{p(\tilde{x})=x}{\Pi}Hom(k,k)$ since the $Hom$ functor sends direct sums to direct products and the first sum is taken over finitely many elements. Therefore we have $Hom(SC^*(W,H),k)\cong SC_*(\tilde{W},k)$ where the latter is the group involved in the definition of symplectic homology of covering space.

    Note that $\tilde{d}(\tilde{x})=\underset{\tilde{y}}{\Sigma}\tilde{\mathcal{M}}({\tilde{x},\tilde{y}})\tilde{y} $ and $\tilde{\delta}(\tilde{y})=\underset{\tilde{x}}{\Sigma}\tilde{\mathcal{M}}({\tilde{x},\tilde{y}})\tilde{x} $ where $\tilde{\mathcal{M}}({\tilde{x},\tilde{y}})$ is the lift of rigid trajectories in $\mathcal{M}(x,y)$ which start from $\tilde{y}$ and end in $\tilde{x}$. Therefore the matrices represented by $\tilde{d}$ and $\tilde{\delta}$ are transpose to each other. Therefore universal coefficients theorem implies that $SH_*(\tilde{W},H)\cong Hom(SH^*(W,H, k[\pi_1]),k)$. 
    \end{proof}

We cannot take limits to get rid of the dependence on the Hamiltonian since in general $Hom(\varprojlim M,N)\ncong\varinjlim Hom(M,N)$ although $Hom(\varinjlim M,N)\cong \varprojlim Hom(M,N)$. A possible workaround is to use a quadratic Hamiltonian $H_q$ and work with filtered symplectic homology and symplectic cohomology in the proof of \ref{adualc}. We should work with a filtration since we do not want to use a direct product in the definition and define filtered  symplectic cohomology $SH^*_{(-\infty,\alpha)}$. The proof of proposition \ref{adualc}  then implies that $SH_*^{(-\infty,\alpha)}(\tilde{W})\cong Hom(SH^*_{(-\infty,\alpha)}(W,k[\pi_1(W)]),k)$.

\begin{rem}\label{action bound}
    We can prove the applications that use the vanishing result \ref{vansc} without involving product structures explicitly, by the above discussion of using quadratic Hamiltonians which gives us the duality $SH_*^{(-\infty,\alpha)}(\tilde{W})\cong Hom(SH^*_{(-\infty,\alpha)}(W,k[\pi_1(W)]),k)$. The idea is to pick the action bound $\alpha$ large enough such that $SH_*^{(-\infty,\alpha)}(\tilde{W})$ vanishes. The existence of this bound $\alpha$ is guaranteed by the fact that we assume $SH_*({W})=0$ so we can increase our action window until the unit element gets killed. Therefore  by duality we have $SH^*_{(-\infty,\alpha)}(W,k[\pi_1(W)])=0$. This suffices for our applications if we use the filtered long exact sequence in symplectic cohomology given by  
\[\begin{tikzcd}
	{...} & {SH_{(-\infty,\epsilon)}^{-*}({W},k[\pi_1])} & {SH^{-*}_{(\epsilon,\alpha)}(W,k[\pi_1])} & {SH_{(-\infty,\alpha)}^{-*}(W,k[\pi_1])} & {...}
	\arrow[from=1-1, to=1-2]
	\arrow[from=1-2, to=1-3]
	\arrow[from=1-3, to=1-4]
	\arrow[from=1-4, to=1-5]
\end{tikzcd}\]

We identify $SC^{-*}({W},k[\pi_1])/SC^{-*}_{(>\epsilon)}({W},k[\pi_1])$ with $\frac{SC^{-*}({W},k[\pi_1])/SC^{-*}_{(>\alpha)}({W},k[\pi_1])}{SC^{-*}_{>\epsilon}({W},k[\pi_1])/SC^{-*}_{(>\alpha)}({W},k[\pi_1])}$ in deriving the above sequence. Note that $SH_{(-\infty,\epsilon)}^{-*}(W,k[\pi_1])\cong H_{*+n+1}(\tilde{W)}$  and $SH^{-*}_{(\epsilon,\alpha)}(W,k[\pi_1(W)])$ is generated by Reeb orbits by the action filtration constraint. This observation is enough to make conclusions regarding $H_{*+n+1}(\tilde{W})$ when symplectic cohomology vanishes and if there are index constraints on Reeb orbits.

\end{rem}

\section{Wrapped Floer homology} \label{3}
In this section we treat the relative version of result that the vanishing of symplectic homology implies the vanishing of symplectic homology with $\mathbb{Z}_2[\pi_1]$ coefficients. The argument of proving the result for wrapped Floer homology follows the same strategy as in the closed case using product structures in lifted Floer homologies. However in the open string version, product structures has not yet been defined in the lifted setting and is more difficult than the closed case. .

 Lagrangian Floer homology of the universal cover   was first introduced by Damian in \cite{damian2012floer}. The wrapped version was later treated in \cite{chantrainefloer} which is defined using bilinearized contact homology. We revisit  the construction in \cite{damian2012floer} and make necessary extensions to derive our result. Before doing so, we give a short review of wrapped Floer homology . More detailed accounts can be found in \cite{abouzaid2010open} and \cite{ritter2013topological} .Note we follow grading conventions as in \cite{lazarev2020contact}.

Let $W^{2n}$ be a Liouville domain with Liouville form $\lambda$ and $L$ be an exact Lagrangian in $W$ with a Legendrian boundary $\Lambda$ contained in the contact manifold $\partial W.$ More precisely by an exact Lagrangian $L$ in $W$ we mean a Lagrangian $L$ such that the Liouville form $\lambda$ restricted to $L$ is exact and the transverse intersection $\partial L=\partial W\pitchfork L$ is Legendrian ,that is $\lambda\vert\partial L=0 $. Furthermore using a Hamiltonian isotopy of L, without loss of generality, we can impose the condition that near the boundary $\partial L$, $\lambda|L$ vanishes.

Let $H$ be a quadratic Hamiltonian  and $J$ an admissible almost complex structure in the setup of section [\cite{ritter2013topological},18.2 ] or \cite{zhou2021symplectic} for the integrated maximum principle to work. We work with $\mathbb{Z}_2$ coefficients as it suffices for all our applications. However all of following theory can be upgraded to $\mathbb{Z}$ coefficients if the Lagrangian ${L}$ has vanishing Maslov class.

Let $\hat{L}:=L\bigcup_{\Lambda}[1,\infty)\times\Lambda$ be the completion of $L$ in the Liouville completion $\hat{W}$. For such a Hamiltonian we denote the Hamiltonian chords of $H$ by $C(H)$ and we define the chain group $CW(L,H,J):=\bigoplus_{\gamma\in C(H)}\mathbb{Z}_2\gamma$. The Hamiltonian chords comes with a $\mathbb{Z}/2\mathbb{Z} $ grading by the Conley Zehnder index $\mu$ of the chords which can be upgraded to $\mathbb{Z}$ grading when the Lagrangian has vanishing Maslov class.

Given two Reeb chords $x_-$ and $x_+$ of the Lagrangian $\hat{L}$, the moduli space of  wrapped trajectories is given by $\mathcal{M}(x_-,x_+):=\{u:\mathbb{R}\times [0,1]\rightarrow \hat{W}, \partial_s u+J(\partial_t u-X_H)=0, lim_{s\rightarrow\pm\infty} u(s,t) =x_\pm, u(s,0)\in\hat{L},u(s,1)\in \hat{L}\}/\mathbb{R}.$ Transversality and gluing implies that under a generic perturbation  $\mathcal{M}(x_+,x_-)$ forms a manifold with corners. Similar to section \ref{2}, the maximum principle ensures that the trajectories do not escape to infinity and that the moduli space is therefore compact. The differential $d:CW(L,H)\rightarrow CW(L,H)$  is defined by $$d(x_+)=\bigoplus_{x_-}|\mathcal{M}_0(x_+.x_-)|x_-$$ where $\mathcal{M}_0(x_+,x_-)$ represents rigid moduli spaces of wrapped trajectories. Similar arguments as in the closed case implies that $d\circ d=0$ and we denote the resulting homology group as   $HW(L)$.

Similar to the scenario in symplectic homology, we can associate an action filtration to the Reeb chords which gives us a filtered subcomplex of the wrapped Floer complex $CW(L,H)^{<k}$. This results in filtered wrapped Floer homology groups $HW(L)^{<k}$. By the choice of our Hamiltonian, the Hamiltonian chords of low action correspond to Morse critical points of $H$ in $L$  and the rest to contractible Reeb chords in the collar. Thereby for a small $\epsilon$, the filtered wrapped Floer homology group $HW(L)^{<\epsilon}$ is isomorphic to the Morse homology group and the subcomplex $CW(L,H)^+:= CW(L,H)/CW(L,H)^{<\epsilon}$ is generated by Reeb chords in the collar. We define the \textbf{positive wrapped Floer homology group} denoted by $HW^+(L)$ as the homology of $CW(L,H)$.By construction we have the following exact triangle.

\begin{prop}
    For an exact $L$ with vanishing Maslov class in a Liouville domain $W^{2n}$, there exists a long exact sequence of groups given by 
   
    \[\begin{tikzcd}
	{...} & {H^{n-k-2}(L)} & {HW_k(L)} & {HW_k^+(L)} & {H^{n-k-1}(L)} & {...}
	\arrow[from=1-1, to=1-2]
	\arrow[from=1-2, to=1-3]
	\arrow[from=1-3, to=1-4]
	\arrow[from=1-4, to=1-5]
	\arrow[from=1-5, to=1-6]
\end{tikzcd}\]

\end{prop}

We can relate symplectic homology to wrapped Floer homology by the following result.
\begin{thm}[\cite{ritter2013topological}] For an exact Lagrangian $L$ in a Liouville domain $W$, the wrapped Floer homology $HW_*(L)$ forms a module over $SH_*(W)$. 
    \end{thm}
\begin{rem}
    The above theorem implies that if symplectic homology vanishes for a manifold, then the wrapped Floer homology would also vanish for any exact Lagrangian by the existence of the unit element in the symplectic homology ring.
\end{rem}

\subsection{Wrapped Floer homology of covering spaces and  lifted TQFT maps.} \label{3.1}
We define wrapped Floer homology $HW(\tilde{L})$ for a  covering space $\tilde{L}\rightarrow \hat{L}$ following the construction  of \cite{damian2012floer} in the compact setting. Note that, another construction of wrapped Floer cohomology of the universal cover is done in  \cite{chantrainefloer} using bilinearized LCH. Another difference from section \ref{2.2}, is that we cannot lift the images of the holomorphic curves since they do not lie in $L$ and therefore we follow \cite{damian2012floer} in lifting the paths defined on the boundary side of the holomorphic strips. The novel part of this section is in the construction of a lifted open pair of pants  in the wrapped Floer homology of the covering space $\tilde{L}$.

Let $p:\tilde{L}\rightarrow \hat{L}$ be a covering of $\hat{L}$. Let $x,y\in C(H)$ be a Hamiltonian chord with endpoints in $ \hat{L}$.
The chain group is given by $$CW^{(0)}_*(\tilde{L}):\bigoplus_x\underset{p(\tilde{x}(0))=x(0)}{\Pi}\mathbb{Z}\tilde{x}(0).$$ To define the differential $\tilde{d}$, consider a wrapped trajectory $u\in\mathcal{M}(y,x)$ and define $\alpha_0(u)$  to be the path given by $u(s,0)$ extended to $[-\infty,\infty]$ in $L$ which starts at $x(0)$ and ends in $y(0)$ .
Let $\tilde{\mathcal{M}}(\tilde{y},\tilde{x})$ be the moduli space of lifted paths $\alpha_0(u)$ for $u\in \mathcal{M}(y,x)$ starting at $\tilde{x}$ and ending at $\tilde{y}$. For a compact rigid moduli space $\mathcal{M}(y,x)$, the  space $\sqcup_{\tilde{y}}\tilde{\mathcal{M}}(\tilde{y},\tilde{x})$ is finite by unique lifting. 
Therefore we can define a differential $\tilde{d}:CW^{(0)}_*(\tilde{L})\rightarrow CW^{(0)}_*(\tilde{L})$ given by $$\tilde{d}(\tilde{y})=\sum_{\tilde{x}}\tilde|{\mathcal{M}}(\tilde{y},\tilde{x})|\tilde{x}.$$ The proof that $d^2=0$ follows identically from the proof of the Lagrangian setting in \cite{damian2012floer}. We denote the homology of  $CW_*^{(0)}(\tilde{L})$ as the wrapped Floer homology of the covering space $HW_*^{(0)}(\tilde{L}).$

\begin{rem}
    
Note that the role of $s=0$ boundary side and $s=1$ side could be swapped in the definitions. In the above description we recover a chain complex $CW^{(1)}_*(\tilde{L})$ whose homology is denoted $HW^{(1)}(\tilde{L})$. However $HW^{(0)}(\tilde{L})$ need not be isomorphic to $HW^{(1)}(\tilde{L})$ unless one makes additional assumptions on the underlying topology (see remark \ref{product with restrictions}).
\end{rem}

Like in the non-lifted case, by filtering the wrapped complex by action, we have the following long exact sequence
\begin{prop}\label{les wcs}
     For a covering $\tilde{L}$ of an exact $L$ (with grading in $\mathbb{Z}_2$ if the Maslov class is not zero) in a Liouville domain $W^{2n}$, there exists a long exact sequence of groups given by 
   
    \[\begin{tikzcd}
	{...} & {H^{n-k-2}(\tilde{L})} & {{HW^{(0)}}_k(\tilde{L})} & {{HW^{(0)}}_k^+(\tilde{L})} & {H^{n-k-1}(\tilde{L})} & {...}
	\arrow[from=1-1, to=1-2]
	\arrow[from=1-2, to=1-3]
	\arrow[from=1-3, to=1-4]
	\arrow[from=1-4, to=1-5]
	\arrow[from=1-5, to=1-6]
\end{tikzcd}\]

\end{prop}

\label{modelriemannsurface}We describe the setup for defining the product and TQFT structures on $HW_*(L)$. For more details regarding the construction we refer to \cite{ritter2013topological} .  
Let $(S,j)$ be the closed unit disk $D$ with $n+1$ punctures  $a_1,...,a_n,b$ on the boundary $\partial D $. Near the punctures $a_i$ we can identify the neighborhoods around them with positive strip-like end $[0,\infty)\times[0,1]$ and around $b$ with negative strip-like ends $(-\infty,0]\times [0,1]$ .Let $\beta$ be a 1-form on $S$ which is equal to $A_idt$ and $ Bdt$  in the strip-like ends of $a_i$ and $b$ respectively, such that $\Sigma_i A_i\geq B , \beta|_{\partial S}=0$ and $\beta $ is exact near $\partial S$ and $d\beta(v,jv)\leq 0$ for all $v\in TS$. In case of the product one can choose $A_1=A_2=1$ and $B=2$, so that Hamiltonians correspond to $H,H$ and $2H$ in the strip like ends.
\begin{rem}
    For the TQFT products we use $\mathbb{Z}_2$ coefficients since it suffices for our main application, however the theory is equally valid over $\mathbb{Z}$ coefficients by considering orientations as in \cite{ritter2013topological} if the Lagrangian is spin.
\end{rem}

We give short recap of TQFT maps following \cite{ritter2013topological}. Let $H$ be a quadratic Hamiltonian and let $\mathcal{M}(x_1,x_2..,x_n;y,H,\beta)$ be the collection of maps $u:S\rightarrow\hat{W}$  such that $u(\partial S)\subset \hat{L}$ and $(du-X_H\otimes\beta)^{0,1}=0$ converging to Hamiltonian chords $x_i,y$ for Hamiltonians $A_iH,BH$ respectively. In essence $(du-X_H\otimes\beta)^{0,1}=0$ reduces to the Floer equation for the Hamiltonians $A_iH,BH$ in the strip like ends of $x_i,y$ respectively. For a generic $J$, the moduli space $\mathcal{M}(x_1,...,x_n,y,H,\beta)$ is cut out transversely and forms a manifold with corners. We use the one form $\beta$  to ensure that a maximum principle works with more than one positive or negative punctures to ensure compactness of the moduli space.

The TQFT product $\psi_S:CW(L,A_1H)\otimes....\otimes CW(L,A_nH)\rightarrow CW(L,BH)$ is defined by the map $\psi(x_1\otimes.....\otimes x_n)=\oplus|\mathcal{M}_0(x_1....,x_n;y,H,\beta)\vert y$ where $\mathcal{M}_0(x_1....,x_n;y,H,\beta)$ denotes a compact rigid moduli space. The product satisfies the Leibniz rule $d(\psi(x_1,....,x_n))=\psi(dx_1,...x_n)+.....\psi(x_1,....,dx_n)$ which corresponds to the breaking of index $1$ curves at the $n+1$ strip-like ends. Therefore $\psi_s$ descends to homology and we get a TQFT product $\psi:HW(L)\otimes......\otimes HW(L)\rightarrow HW(L)$.

We can now proceed with the construction of a lifted TQFT map $\tilde{\psi_S}$ on $HW^{(0)}(\tilde{L})$. Let $\gamma_0$ denote the path from $a_1$ to $b$ on $\partial S$. For $u\in$ $\mathcal{M}(x_1,...,x_n;y,H,\beta)$, we denote by $u(\gamma_0)$ the image of $\gamma_0$ in $\hat{L}$ extended to $[-\infty,\infty]$. Note that $u(\gamma_0)$ is a path from $x_1(0)$ to $y(0)$.  

Given a compactified moduli space $\mathcal{M}(x_1,....,x_n;y,H,\beta)$ we  denote by $\mathcal{M}_{\widetilde{x_1(0)},\widetilde{y(0)}}(x,y,z)$ to be the moduli space comprising the lifts of $u(\gamma_0)$ that start at $\widetilde{x_1(0)}$ and end at $\widetilde{y(0)}$. We define a map $\tilde{\psi_S}$ given by $$\tilde{\psi_S}:{CW}^{(0)}(\tilde{L},A_1H)\otimes ...\otimes CW(L,A_nH)\rightarrow \tilde{CW^{(0)}}(\tilde{L},BH)$$
$$\tilde{\psi_S}(\widetilde{x_1(0)},...,x_n)=\Sigma\vert\mathcal{M}_{\widetilde{x_1(0)},\widetilde{y(0)}}(x_1,....,x_n;y,H,\beta)\vert\widetilde{y(0)} $$

\begin{thm}
    The map $\tilde{\psi_S}:CW^{(0)}(\tilde{L},A_1H)\otimes ...\otimes CW(L,A_nH)\rightarrow \tilde{CW^{(0)}}(\tilde{L},BH)$ induces a well defined map $\Psi_S:{HW}^{(0)}(\tilde{L})\otimes .....\otimes HW(L)\rightarrow {HW}^{(0)}(\tilde{L})$ at the level of homology. 
\end{thm}

\begin{proof}
    The map $\tilde{\psi_S}$ induces a map at the level of homology if we can ensure that the Leibniz rule $$\tilde{\psi_S}(\tilde{d}\tilde{x_0},...,x_n)+\tilde{\psi_S}(\tilde{x_0},dx_n)=\tilde{d}\tilde{\psi_S}(\tilde{x_0},...,x_n)$$ holds in the chain level. We prove the above statement in the case of the two positive punctures $x_1$, $x_2$ and a negative puncture $y$ since the same proof will directly generalize to case of $n$ positive punctures. 
    
     After a generic perturbation of $J$, let $\mathcal{M}_1(x_1,x_2,y,H,\beta)$ be the moduli space of index 1 curves, which is a 1-dimensional manifold by standard transversality arguments. The boundary of $\overline{\mathcal{M}_1(x_1,x_2,y,H,\beta)}$ is given by the breaking of strip-like ends where we denote the broken moduli spaces at the strip-like end of $x_1,x_2$ and $z$ by $T_1(x_1,x_2,y), T_2(x_1,x_2,y)$ and $T_3(x_1,x_2,y)$   respectively (see figure \ref{fig}). 

    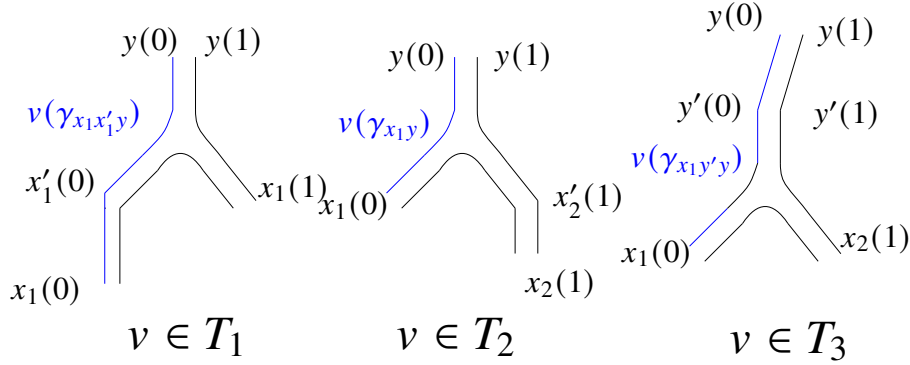
\begin{figure}
    \centering

\begin{tikzpicture}[scale=10, transform shape]

\draw[ultra thin, blue, xshift=-8pt] (1.37,3.5) -- (1.37,3.43);
  \draw[ultra thin, xshift=-8pt] (1.4,3.5) -- (1.4,3.43);
  
\draw[ultra thin, blue, xshift=-8pt] (1.28,3.32) -- (1.33,3.37);
\draw[ultra thin, xshift=-8pt] (1.3,3.3) -- (1.35,3.35);

  \draw[ultra thin, xshift=-8pt] (1.41,3.35) -- (1.45,3.3);
  \draw[ultra thin, xshift=-8pt] (1.44,3.36) -- (1.48,3.31);
 
\draw[ultra thin, blue, xshift=-8pt] (1.33,3.37) .. controls (1.36,3.4) and (1.36,3.4) .. (1.37,3.43);
  \draw[ultra thin, xshift=-8pt] (1.4,3.43) .. controls (1.4,3.4) and (1.41,3.4) .. (1.44,3.36);
  \draw[ultra thin, xshift=-8pt] (1.41,3.35) .. controls (1.39,3.38) and (1.37,3.38) .. (1.35,3.35);

  \draw[ultra thin, xshift=-8pt] (1.3,3.3) -- (1.3,3.2);
  \draw[xshift=-8pt] (1.28,3.3) -- (1.28,3.3);
  \draw[ultra thin, blue, xshift=-8pt] (1.28,3.32) -- (1.28,3.2);
  
\node[xscale=0.1, yscale=0.1, xshift=-8pt] at (0.95,3.19) {$x_1(0)$};
\node[xscale=0.1, yscale=0.1, xshift=-8pt] at (0.97,3.33) {$x_1'(0)$};
\node[xscale=0.1, yscale=0.1, xshift=-8pt] at (1.44,3.5) {$y(0)$};
\node[xscale=0.1, yscale=0.1, xshift=-8pt] at (1.63,3.2) {$x_2(1)$};
\node[xscale=0.1, yscale=0.1, xshift=-8pt] at (1.58,3.5) {$y(1)$};

\node[xscale=0.1, yscale=0.1, text=blue, xshift=-8pt] at (1,3.42) {$v(\gamma_{x_1x_1'y})$};

\node[xscale=0.15, yscale=0.15, xshift=-8pt] at (1.15,3.13) {$v \in T_1$};

  \draw[ultra thin, blue, xshift=2.6pt] (1.37,3.5) -- (1.37,3.43);
  \draw[ultra thin, xshift=2.6pt] (1.4,3.5) -- (1.4,3.43);
  
\draw[ultra thin, blue, xshift=2.6pt] (1.28,3.32) -- (1.33,3.37);
\draw[ultra thin, xshift=2.6pt] (1.3,3.3) -- (1.35,3.35);

  \draw[ultra thin, xshift=2.6pt] (1.41,3.35) -- (1.45,3.3);
  \draw[ultra thin, xshift=2.6pt] (1.44,3.36) -- (1.48,3.31);
 
\draw[ultra thin, blue, xshift=2.6pt] (1.33,3.37) .. controls (1.36,3.4) and (1.36,3.4) .. (1.37,3.43);
  \draw[ultra thin, xshift=2.6pt] (1.4,3.43) .. controls (1.4,3.4) and (1.41,3.4) .. (1.44,3.36);
  \draw[ultra thin, xshift=2.6pt] (1.41,3.35) .. controls (1.39,3.38) and (1.37,3.38) .. (1.35,3.35);
  
\draw[ultra thin, xshift=2.6pt] (1.45,3.3) -- (1.45,3.24);
  \draw[ultra thin, xshift=2.6pt] (1.48,3.31) -- (1.48,3.24);

\node[xscale=0.1, yscale=0.1, xshift=2.6pt] at (1.72,3.24) {$x_1(0)$};
\node[xscale=0.1, yscale=0.1, xshift=2.6pt] at (1.05,3.52) {$y(0)$};
\node[xscale=0.1, yscale=0.1, xshift=2.6pt] at (1.16,3.52) {$y(1)$};
\node[xscale=0.1, yscale=0.1, xshift=2.6pt] at (2.01,3.26) {$x_2(1)$};
\node[xscale=0.1, yscale=0.1, text=blue, xshift=2.6pt] at (1.36,3.41) {$v(\gamma_{x_1y})$};
\node[xscale=0.1, yscale=0.1, xshift=2.6pt] at (1.63,3.31) {$x_2'(1)$};
\node[xscale=0.15, yscale=0.15, xshift=2.6pt] at (1.45,3.13) {$v \in T_2$};

\draw[ultra thin, blue, xshift=14pt, yshift=-2pt] (1.37,3.5) -- (1.37,3.43);
  \draw[ultra thin, xshift=14pt, yshift=-2pt] (1.4,3.5) -- (1.4,3.43);
  
\draw[ultra thin, blue, xshift=14pt, yshift=-2pt] (1.28,3.32) -- (1.33,3.37);
\draw[ultra thin, xshift=14pt, yshift=-2pt] (1.3,3.3) -- (1.35,3.35);

  \draw[ultra thin, xshift=14pt, yshift=-2pt] (1.41,3.35) -- (1.45,3.3);
  \draw[ultra thin, xshift=14pt, yshift=-2pt] (1.44,3.36) -- (1.48,3.31);
 
\draw[ultra thin, xshift=14pt, yshift=-2pt] (1.33,3.37) .. controls (1.36,3.4) and (1.36,3.4) .. (1.37,3.43);
  \draw[ultra thin, xshift=14pt, yshift=-2pt] (1.4,3.43) .. controls (1.4,3.4) and (1.41,3.4) .. (1.44,3.36);
  \draw[ultra thin, xshift=14pt, yshift=-2pt] (1.41,3.35) .. controls (1.39,3.38) and (1.37,3.38) .. (1.35,3.35);
  \draw[ultra thin, blue, xshift=14pt, yshift=-2pt] (1.37,3.5) -- (1.4,3.6);
  \draw[ultra thin, xshift=14pt, yshift=-2pt] (1.4,3.5) -- (1.43,3.6);

\node[xscale=0.1, yscale=0.1, xshift=14pt, yshift=-2pt] at (1.28,3.31) {$x_1(0)$};
\node[xscale=0.1, yscale=0.1, xshift=14pt, yshift=-2pt] at (1.75,3.44) {$y'(0)$};
\node[xscale=0.1, yscale=0.1, xshift=14pt, yshift=-2pt] at (1.93,3.43) {$y'(1)$};
\node[xscale=0.1, yscale=0.1, xshift=14pt, yshift=-1.1pt] at (1.2,3.33) {$x_1(1)$};
\node[xscale=0.1, yscale=0.1, text=blue, xshift=14pt, yshift=-2pt] at (1.72,3.37) {$v(\gamma_{x_1y'y})$};

\node[xscale=0.15, yscale=0.15, xshift=14pt, yshift=-3pt] at (1.83,3.14) {$v \in T_3$};
\node[xscale=0.1, yscale=0.1, xshift=200pt, yshift=-2pt] at (1.13,3.56) {$y(0)$};
\node[xscale=0.1, yscale=0.1, xshift=14pt, yshift=-2pt] at (1.93,3.54) {$y(1)$};

\end{tikzpicture}

    \caption{Breaking at strip-like ends\label{fig}}
    \label{fig:placeholder}
\end{figure}

   Recall that $u(\gamma_0)$ denotes the path from $x_1(0)$ to $y(0)$. For $u$ in  $\mathcal{M}_1(x_1,x_2,y,H,\beta) $, we call the smooth path $u(\gamma_0)$ from $x_1(0)$ to $y(0)$ as $u(\gamma_{x_1y})$. Let $v\in T_1(x_1,x_2;y)$ be a broken pair of pants where the holomorphic strip in the end of $x_1$ breaks at $x_1'$. For $v$ denote by $v(\gamma_0)$ or $v(\gamma_{x_1x_1'y})$, the broken path in $L$ obtained from $v$ as the concatenation on the side $t=0$  from $x_1(0)$ to  $y(0)$ with a breaking at $x_1'(0)$. 

    \begin{lem}[\cite{damian2012floer} Lemma 2.7]\label{damian lemma} For a sequence $\{v_n\}_{n\in \mathbb{N}} \in \mathcal{M}_1(x_1,x_2;y)$ converging to the broken pair of pants $v$, for $n$ large enough the paths $v_n(\gamma_{x_1y})$ are homotopic to the broken path $v(\gamma_{x_1x_1'y})$
        \end{lem}

    In \cite{damian2012floer}, the above lemma is stated for holomorphic strips between  chords and not in the case of the pair of pants. 
    However since the breaking in our case happens in the holomorphic strip-like end, the theorem applies without much modification . Without loss of generality the lemma also holds for $v\in T_3(x_1,x_2,y)$ where the broken path is denoted by $v(\gamma_{x_1y'y})$ (see Figure \ref{fig}).
 For the homotopy class $[\gamma]$ define 
 $$\mathcal{M}_1(x_1,x_2,y,[\gamma]):=\{u\in \mathcal{M}_1(x_1,x_2,y,H,\beta):[u(\gamma_0)]=[\gamma]\}.$$ The paths $u(\gamma_0)$ for $u$ in a connected component of $\mathcal{M}(x_1,x_2,y)$ are homotopic and the previous lemma implies that each broken curve $v$ in $\overline{\partial\mathcal{M}(x_1,x_2,y)}$ occurs as the boundary of $\overline{\mathcal{M}(x_1,x_2,y,[v(\gamma_0)])}$. This implies that the Leibniz rule for the product $\psi_S$ can be proven by partitioning $\overline{\mathcal{M}_1(x_1,x_2,y,H,\beta)}$ based on the homotopy class of the path $u(\gamma_0)$ and showing that the  boundaries adds up to zero in each class that is $$|\partial\overline{\mathcal{M}(x_1,x_2,y,[\gamma])}|=0.$$

 We now can use this to prove the Leibniz rule $$\tilde{\psi_S}(\tilde{d}\widetilde{x_1(0)},x_2)+\tilde{\psi_S}(\widetilde{x_1(0)},dx_2)+\tilde{d}\tilde{\psi_S}(\widetilde{x_1(0)},x_2)=0$$ holds. Let $\tilde{y_0}$ be a lift of $y(0)$ and let $n(\widetilde{x_1(0)},\widetilde{y(0)})$ be the coefficient of  $\widetilde{y(0)}$ in $\tilde{\psi_S}(\tilde{d}\widetilde{x_1(0)},x_2)+\tilde{\psi_S}(\widetilde{x_1(0)},dx_2)+\tilde{d}\tilde{\psi}(\widetilde{x_1(0)},x_2)$. Note that fixing $\widetilde{x_1(0)}$ and $\widetilde{y(0)}$ fixes a homotopy class say $[\gamma]$ of paths between $x_1$ and $y$ in the base. The contribution of $\tilde{\psi_S}(\tilde{d}\widetilde{x_1(0)},x_2)$ to $n(\widetilde{x_1(0)},\widetilde{y(0)})$ is given by $T_1(x_1,x_2,y)\cap \overline{\mathcal{M}(x_1,x_2,y,[\gamma])}$, that of $\tilde{\psi}(\widetilde{x_1(0)},dx_2)$ is given by $T_2(x_1,x_2,y)\cap \overline{\mathcal{M}(x_1,x_2,y,[\gamma])}$ and of $\tilde{d}\tilde{\psi_S}(\widetilde{x_1(0)},x_2)$ is given by $T_3(x_1,x_2,y)\cap \overline{\mathcal{M}(x_1,x_2,y,[\gamma])}$. Then $$n(\widetilde{x_1(0)},\widetilde{y(0)})=|T_1(x_1,x_2,y)\cap \overline{\mathcal{M}(x_1,x_2,y,[\gamma])}|+|T_2(x_1,x_2,y)\cap \overline{\mathcal{M}(x_1,x_2,y,[\gamma])}|+$$ $$|T_3(x_1,x_2,y)\cap \overline{\mathcal{M}(x_1,x_2,y,[\gamma])}|=|\partial \overline{\mathcal{M}(x_1,x_2,y,[\gamma])}|=0.$$ Therefore $\psi_S$ induces a well defined map in homology.
\end{proof}

We adapt the "gluing=composition" result of \cite{ritter2013topological}[Section 16.10], to what we need in the lifted setting. Recall the setting \ref{modelriemannsurface} where $(S,\beta)$ is a Riemann surface with $n+1$ punctures and $\beta$ is a 1-form with weights $A_1,...A_n,B$ at the strip-like ends of the punctures. We denote similarly $S'$ to be a punctured disk with $n'+1$ punctures, with $n'$ positive punctures and one negative puncture. We choose a 1 form $\beta'$ satisfying the properties of our model setup with weights $A'_1,...A'_{n'},B'$ (\ref{modelriemannsurface}). We make the additional assumption that $B'=A_k$ for $k$ since we would like the glue the negative end of $S'$ to a positive end of $S$. From the construction in \cite{ritter2013topological}[Section 16.10][Theorem 16.12], after gluing $S$ to $S'$ with a parameter $\lambda$ at the puncture $k$ to obtain a surface $S\#_\lambda S'$ and attaching trivial cylinders (that induce identity by continuation map) at the other ends of $S$, we obtain that $\psi_{S\#_{\lambda}S'}=\psi_S\circ\psi_{S'}$ for $\lambda>>0$.

Note that the proof of lemma \ref{damian lemma} in \cite{damian2009constraints}[lemma 3.16] only uses the general framework of breaking and can be adapted without much modifications. This implies that given holomorphic curves $v_\lambda$ that contribute to the coefficient of $y$ in $\psi_{S\#_{\lambda}S'}(x_1,...,x_{n'})$ and the broken curve $v\#v'$ which contributes  to the coefficient of $y$ in $\psi_S\circ\psi_{S'}(x_1,....,x_{n'})$,  the paths evaluated from $x_1$ to $y$ denoted by $v_\lambda(\gamma _{x_1 y})$ and $v\#v'(\gamma_{x_1x_1'y})$ are homotopic for $\lambda$ large enough. This allows us to conclude the following corollary.

\begin{cor}["Gluing=Composition"] \label{gluing=composition}
    For ${S\#_{\lambda}S'}$ and $\lambda$ large enough, $\tilde{\psi}_{S\#_{\lambda}S'}=\tilde{\psi}_S\circ\tilde{\psi}_{S'}$ for lifted TQFT operations.
\end{cor}

\begin{thm}[Theorem \ref{mod}] \label{module}
    The lifted wrapped Floer homology of a covering $\tilde{L}$, $ HW^{(0)}(\tilde{L})$ forms a module over the the wrapped homology of the base $HW(L)$.
\end{thm}
\begin{proof}
   Let $(P,j,\beta)$ be a pair of pants and , then using corollary \ref{gluing=composition} we have $${\Psi}_P(\widetilde{x_1(0)},x_2.x_3)=\Psi_P(\widetilde{x_1(0)},\Psi_{P'}(x_2,x_3))={\Psi}_{P\#_\lambda P'}(x_1(0),x_2,x_3)$$ for $P'$ another model open pair of pants whose negative end can be glued to  the positive end of P. Since the choice of TQFT operations $\Psi_S$ do not depend on $j,\beta,J$, we have proved associativity of the module structure. 
 Note that for  model Riemann surface $E$ with one positive puncture and no negative punctures  induces the TQFT map $$\psi_E:\mathbb{Z}_2\rightarrow HW(L).$$ The unit $e$ is given by  $\psi_E(1)$. Therefore again using the "gluing=composition" lemma \ref{gluing=composition} , we have ${\Psi_P}(\widetilde{x_1(0)},e)={\Psi}_{P\#_\lambda E}(\widetilde{x_1}(0))=\widetilde{x_1}(0).$ The last equality follows since after gluing we get a continuation cylinder at the chain level which corresponds to identity at homology.
\end{proof}

\begin{rem}
    One should also be able to prove that ${HW}^{(0)}(\tilde{L})$ forms a module over $HW(L)$ by adapting the lifting arguments to the setting of the $A_\infty$ product $\mu_3$ in the wrapped Fukaya category of $L$ by varying the conformal structure in the domain in contrast to a fixed $j$ in the TQFT approach which introduces additional breakings correponding to the boundary of the Deligne-Mumford moduli space, to prove the associativity of the product in $HW(L)$(cf \cite{auroux2013beginner}).
\end{rem}

We can now deduce our main vanishing result.
\begin{cor}\label{Vanishing WH}
    Let $L$ be an exact Lagrangian in a Liouville domain $W$ such that the wrapped Floer homology of $L$, $HW(L)$ vanishes. Then the lifted wrapped Floer homology of the a covering $\tilde{L}$, $HW^{(0)}(\tilde{L})$ vanishes. 
\end{cor}
\begin{proof}
    From \ref{module}, we know there exists a unit element $e$ in $HW(L)$ which acts as identity on $HW^{(0)}(\tilde{L})$. Therefore vanishing of $HW(L)$ forces $e=0$ and therefore $HW^{(0)}(\tilde{L})=0$.
\end{proof}
\begin{rem}\label{product with restrictions}
    Note that if group $\pi_2(W,L)=0$, one can show that $CW^{(0)}(\tilde{L})$ is isomorphic to $CW^{(1)}(L)$ at the chain level since  the paths induced on either side of the strip are homotopic in $W$ and hence in $L$. After identifying the groups one can similarly use the homotopy induced by open pair of pants to define a product structure on $HW^{(0)}(\tilde{L})\cong HW^{(1)}(\tilde{L})$ and in more generality $A_\infty $ products at chain level. 
\end{rem}

\section{Applications}\label{4}
We prove the theorems in the introduction using the vanishing results of the previous sections.
\subsection{Dynamical Convexity and topology} \label{4.1}
Kwon-Oba proved in \cite{kwon2024dynamically}  that given a topologically simple Liouville filling $W$ with vanishing $SH_*(W)$ of a  dynamically convex contact manifold, the $W$
   is a homology ball, that is $H_*(W,\partial W)=H_*(D^{2n},\partial D^{2n})$. We use both symplectic homology and symplectic cohomology to get similar information on  $\partial W$ which does not depend on the filling $W.$ 

\begin{thm}[Theorem \label{hom} \ref{homotop}] Let $M^{2n-1}$ be a topologically simple dynamically convex contact manifold that admits an exact filling $W$ with vanishing $SH_*(W)$. Then $M$ is a homology sphere.
    
\end{thm}

\begin{proof}
    Let $k$ be an arbitrary coefficient field. Dynamical convexity implies  $SH_i^+(W,k)=0$ and $SH^i_+(W,k)=0$ for $i<n+1$ and using exactness of the long exact sequence of symplectic homology (Proposition \ref{les}) with the vanishing $SH_*$ assumption forces $H_i(W,\partial W,k)=0$ for $i<2n.$ Moreover $SH_*(W,k)=0$ implies $SH^*(W,k)=0$ by Theorem \ref{adual}. A similar exactness argument using the long exact sequence of symplectic cohomology (Proposition \ref{les2}) implies that $H_i(W,k)=0 $ for $i>0$. The long exact sequence of homology $$...\rightarrow H_i(\partial W,k)\rightarrow H_i(W,k)\rightarrow H_i(W,\partial W)\rightarrow H_{i-1}(\partial W,k)\rightarrow ...$$ implies $\partial W$ is a homology sphere over arbitrary field coefficients. Universal coefficients theorem thereby implies $\partial W$ is a homology sphere over $\mathbb{Z}$ since the homology is finitely generated as $\partial M$ is a compact manifold.
\end{proof}

\begin{thm} \label{4.2}
    Let $M^{2n-1}$ be a dynamically convex contact manifold such that there exists an exact filling $W$ of $M$ with $SH_*(W)=0$ and $\pi_1(W)$ is finite. Then, $\pi_k(W)\cong \pi_k(M)$ for $1<k<2n-1$.
\end{thm}

  \begin{proof}

      Let $\tilde{W}$ denote the universal covering of $W$.Since $M$ is dynamically convex, $SH_k^+(W,\mathbb{Z}[\pi_1])=0$ for $k\leq n$ since there are no generators. The condition $\pi_1(W)$ is finite implies $\tilde{W}$ is compact. After using Lefshetz duality to identify the  cohomology ${H^k}(\tilde{W})$ and $H_{2n-k}(\tilde{W},\partial \tilde{W})$, the canonical long exact sequence with $\mathbb{Z}[\pi_1]$ coefficients is given by

  \[\begin{tikzcd}[cramped]
	{...} & {H_{*+n}(\tilde{W},\partial \tilde{W})} & {SH_{*}(W,\mathbb{Z}[\pi_1])} & {SH_{*}^{+}(W,\mathbb{Z}[\pi_1])} & {H_{*+n-1}(\tilde{W},\partial \tilde{W})} & {...}
	\arrow[from=1-1, to=1-2]
	\arrow[from=1-2, to=1-3]
	\arrow[from=1-3, to=1-4]
	\arrow[from=1-4, to=1-5]
	\arrow[from=1-5, to=1-6]
\end{tikzcd}\] 
    
$SH(W)=0$ implies $SH(W,\mathbb{Z}[\pi_1])=0$ by corollary \ref{loccoeff}.  The long exact sequence implies ${H_{k}(\tilde{W},\partial \tilde{W})}=0$ for $k<2n$. By examining the part $H_1(\tilde{W},\partial \tilde{W}) \rightarrow H_0(\tilde{W})\rightarrow H_0(\partial\tilde{W})$   of the long exact sequence of homology groups, we can conclude that $H_0(\tilde{W})\rightarrow H_0(\partial\tilde{W})$ is injective since $H_1(\tilde{W},\partial \tilde{W})=0$. Therefore injectivity forces $H_0(\partial\tilde{W})=\mathbb{Z}$ since $H_0(\tilde{W})=\mathbb{Z}.$ We conclude that the boundary $\partial\tilde{W}$ is connected.

      Examining the part of long exact sequence of the relative homotopy groups of the pair $(\partial \tilde{W}, \tilde{W}) $ given by  $\rightarrow \pi_1(\tilde{W})\rightarrow\pi_1(\tilde{W},\partial\tilde{W})\rightarrow\pi_0(\partial\tilde{W})\rightarrow \pi_0(\partial W) $, we deduce that $\pi_1(\tilde{W},\partial\tilde{W})=0.$

 Relative Hurewicz theorem implies ${\pi_{k}(\tilde{W},\partial \tilde{W})}=0$ for $k<2n$. This implies $\pi_k(\tilde{W})\cong \pi_k(\partial \tilde{W})$ for  $k<2n-1$ using the long exact sequence of relative homotopy groups. Since projection from a covering space induces isomorphisms on $\pi_k$ for $k\geq 2$, we get $\pi_k(W)\cong \pi_k(M)$ for $1<k<2n-1$.
      \end{proof}  
  We prove similarly to theorem \ref{hom} by using symplectic homology and cohomology with local coefficients ,that for dynamically convex contact manifolds $M$ that admit a filling with vanishing symplectic homology,  $\pi_k(M)=0$ for $1<k<2n-1.$

  \begin{thm}[Theorem \ref{homotop}] \label{homotopy}
      Let $M^{2n-1}$ be a dynamically convex contact manifold that admits an exact filling $W$ with vanishing symplectic homology, then $\pi_i(W)=0$ for $i>1$. If $W$ has a finite fundamental group, then $\pi_k(M)=0$ for $1<k<2n-1$.
  \end{thm}
\begin{proof}
    Let $k$ be an arbitrary field. By universal coefficients, we know that $SH_*(W,\mathbb{Z})=0\implies SH_*(W,k)=0.$ The algebraic duality statement \ref{adual} implies $SH^*(W,k)=0.$ This further implies that $SH^*(W,k[\pi_1(W)])=0$ by \ref{vansc}. Now because of dynamical convexity  there are no generators of $SH_+^*(W,k[\pi_1(W)])$ of index less than $n+1$ implying $SH_+^i(W,k[\pi_1(W)])=0$ for $i<n+1$. The long exact sequence \ref{lesloc2} implies that $H_i(\tilde{W},k)=0$ for $i>0$ where $\tilde{W}$ is the universal cover of $W$. 

    Note that $H_i(\tilde{W})$ need not be finitely generated say if $\pi_1(W)$ is infinite. By splitting of the  universal coefficients  short exact sequence we have $H_i(\tilde{W},\mathbb{Z})\otimes k\bigoplus Tor(H_{i-1}(\tilde{W}),k)\cong H_i(\tilde{W},k)$. For $i=1$,we have $H_1(\tilde{W},\mathbb{Z})\otimes k\cong H_1(\tilde{W},k)=0$ since $Tor(\mathbb{Z},k)=0.$ From universal coefficients with $i=2$, we have $Tor(H_1(\tilde{W},\mathbb{Z}),k)=0$. Now $Tor(H_1(\tilde{W},\mathbb{Z}),k)=0$ and $H_1(\tilde{W},\mathbb{Z})\otimes k=0$ implies $H_1(\tilde{W},\mathbb{Z})=0.$ By induction we have $H_i(\tilde{W})=0$ for $i>0$.

For the universal cover $\tilde{W}$, by Hurewicz $\pi_2(W)\cong H_2(\tilde{W})=0$ and by induction $\pi_i(W)\cong0$ for $i>1$. Theorem \ref{4.2} implies that $\pi_i(W)\cong \pi_i(M)$ for $1<i<2n-1$. Therefore we have $\pi_i(M)=0$ for $1<i<2n-1$ since $\pi_i(\tilde{W})\cong \pi_i(W)$ for $i>1$.
    \end{proof}

\begin{rem}
    The above proofs can be adapted to working with index bounds that is not necessarily $n+1$ of dynamical convexity. An interesting source of lower bounds on Conley Zehnder index comes from the minimal discrepancy number of the link of a singularity cf \cite{mclean2016reeb}. One can impose restrictions on fillings of its link from the minimal discrepancy number of a singularity.
\end{rem}

\subsection{Conley-Zehnder indices and topology} We proved in the previous subsection that dynamical convexity and vanishing of symplectic homology of a filling force strong topological restrictions on the contact manifold. We would like to show here  that the general idea of working with both symplectic cohomology and homology of any strong filling can still be used to reveal topological information of the contact manifold from Reeb dynamics by weakening the hypothesis of dynamical convexity. 

Throughout this subsection, we assume that the contact manifolds admits a strong aspherical fillings. This induces a $\mathbb{Z}$ grading on symplectic (co)homology. 

\begin{thm}[Theorem \ref{czi}]\label{cz}
    Let $M^{2n-1}$ be a contact manifold such that it admits an exact filling $W$ with $SH_*=0$. If there exists a non-degenerate contact form $\alpha$ on $M$ with no Reeb orbits of Conley Zehnder index $2+k$ and $1-k$, then $H_{k+n}(\partial W)=0$ and $H_{n-k-1}(\partial W,\mathbb{Z})=0$.
\end{thm}
\begin{proof}
    Our assumption that there exists no Reeb orbit of index $2+k$ and $1-k$ implies that the positive symplectic homology groups $SH^+_{2+k}(W)$ and $SH^+_{1-k}(W)$ vanish since these groups are generated by Reeb orbits of index $2+k$ and $1-k$ respectively. Using the long exact sequence \ref{les} and the fact that $SH_*(W)=0$ we get that $H_{n+k+1}(W,\partial W)=0$ and $H_{n-k}(W,\partial W)=0.$
    
    Working with arbitrary field coefficients $k$, we can similarly get information about $H_*(W,k)$. Absence of Reeb orbits of index $2+k$ and $1-k$ implies that the  positive symplectic cohomology group $SH_+^{2+k}(W,k)$ and $SH_{+}^{1-k}(W,k)$ vanishes. Furthermore vanishing of symplectic homology $SH_*(W,k)$ forces $SH^*(W,k)=0$ by \ref{adual}. Therefore from the exactness of the long exact sequence in symplectic cohomology \ref{les2}, we deduce $H_{n+k}(W,k)$ and $H_{n-k-1}(W,k)=0$ for all field coefficients $k$. Since $W$ is compact , the homology groups $H_*(W,\mathbb{Z})$ are finitely generated. We can conclude using universal coefficients theorem, that $H_{n+k}(W,\mathbb{Z})$ and $H_{n-k-1}(W,\mathbb{Z})=0.$ 

The long exact sequence of relative homology groups of the pair $(W,\partial W)$ is given by \[\begin{tikzcd}
	{...} & {H_{k+n+1}(W,\partial W)} & {H_{k+n}(\partial W)} & {H_{k+n}(W)} & {....}
	\arrow[from=1-1, to=1-2]
	\arrow[from=1-2, to=1-3]
	\arrow[from=1-3, to=1-4]
	\arrow[from=1-4, to=1-5]
\end{tikzcd}\]. The exactness of the above sequence implies $H_{k+n}(\partial W)=0$ and $H_{n-k-1}(\partial W,\mathbb{Z})=0.$
\end{proof}
\begin{rem}
    Note that  one can further weaken the hypothesis of vanishing symplectic homology to vanishing of symplectic homology in particular degrees for instance $SH_{1+k}(W),SH_{-k}(W)$ to force $H_{n+k}(\partial W)=0$  to hold.
\end{rem}
\begin{rem}
    The theorem \ref{cz} implies that if $H_{k+n}(\partial W)\neq 0$ or $H_{n-k-1}(\partial W,\mathbb{Z})\neq 0$ then there necessarily exists a Reeb orbit of Conley Zehnder index $2+k$ or $1-k$. 
\end{rem}

\subsection{Lagrangian fillings of the standard Legendrian sphere}\label{4.3}
In this section, using the vanishing result theorem \ref{Vanishing WH}, we prove the main application of the paper. Note that by a lacunary contact form, we mean a form where the degree of  the Reeb chord generators are either all odd or all even.
\begin{thm}\label{funda}
    Let $L$ be an exact Lagrangian filling with a Maslov zero class of a Legendrian $\Lambda$ in a Liouville domain $W$ with vanishing wrapped Floer homology $HW(L)$. If $\Lambda$ admits a lacunary contact form, then $\pi_1(L)$ is trivial.
\end{thm}
\begin{proof}
 The vanishing result of corollary \ref{Vanishing WH} implies that $HW^{(0)}(\tilde{L})=0$ for the universal cover $\tilde{L}$ of $L$. From the long exact sequence \ref{les wcs}, we can conclude that we have an isomorphism ${HW_{n-1}^0}^+(\tilde{L})\cong H^0(\tilde{L}) $. Since we assume that the form is lacunary, the differentials in ${HW_{k}^{(0)}}$ vanish and ${HW_{n-1}^{(0)}}$ is generated by Reeb chords of which there exists at least one since its non-zero. However $H^0(\tilde{L})\cong \mathbb{Z}$ whereas ${HW_{n-1}^{(0)}}^+(\tilde{L})={CW^{(0)}_{n+1}}^+(\tilde{L})= \bigoplus_{\mu_{CZ}(x)=n+1}\underset{p(\tilde{x}(0))=x(0)}{\Pi}\mathbb{Z}\tilde{x}(0)\cong \bigoplus_{\mu_{CZ}(x)=n+1}\mathbb{Z}[[\pi_1(L)]]$. Therefore the isomorphism ${HW_{n-1}^{(0)}}^+(\tilde{L})\cong H^0(\tilde{L}) $ implies $\pi_1(L)=1$.
\end{proof}
\begin{rem}
    In case the filling is not Maslov zero and there exists only one single transverse Reeb chord as in \cite{chantrainefloer}, the above proof works over $\mathbb{Z}_2$ grading. However this is not be the case for the standard Legendrian $\Lambda_{std}$ in $(S^{2n-1},\xi_{std})$, as there are infinitely many Reeb chords for a non-degenerate contact form and the form need not be lacunary. (\cite{kim2024topology},Example 2.6).
\end{rem}

Before proving theorem \ref{lagrangian filling}, we state an invariance result for wrapped Floer homology of universal covers. Let $\alpha$ be contact form on $M$ and $\Lambda$, a Legendrian submanifold in $M$. We denote by $\mathcal{C}^{>D}(\Lambda,\alpha)$ to be the set of all  Reeb chords of $\Lambda$ in $Y$ that are contractible relative to $\Lambda$ and have action bounded below by $D$.

\begin{defn}[Lazarev \cite{lazarev2020contact}] Let $\Lambda$ be a Legendrian in a contact manifold $(Y,\xi)$ such that $c_1(Y,\mathbb{Q})=0$, $\pi_1(Y,\Lambda)=0$ and the Maslov class $\mu_{\Lambda}(Y,\Lambda)$ vanishes. $\Lambda$ is asymptotically dynamically convex if there exists an infinite sequence of contact forms $\alpha_1\geq\alpha_2...$ for $Y$ and an infinite sequence of positive numbers $D_1<D_2<....
$ tending to infinity, such that for any $\gamma\in\mathcal{C}^{>D_i}(\Lambda,\alpha_i)$, $\gamma$ is non-degenerate and index $\mu_{CZ}(\gamma)>1$.
    
\end{defn}

We have an invariance result for positive wrapped Floer homology which is proved in a similar way in the case of positive symplectic homology.

\begin{prop}[\cite{lazarev2020contact},proposition 7.10]\label{adcleg}
    Let $\Lambda$ be an asymptotically dynamically convex Legendrian in an asymptotically dynamically contact manifold $Y$. For two exact Maslov-zero Lagrangians $L_1$ and $L_2$ in Liouville fillings $W_1$ and $W_2$ of $M$, then $HW^+(L_1)\cong HW^+(L_2)$ 
\end{prop}

\begin{prop}\label{adccover}
    Let $\Lambda$ be an asymptotically dynamically convex Legendrian in an asymptotically dynamically convex contact manifold $Y$. For two exact Maslov-zero Lagrangians $L_1$ and $L_2$ in Liouville fillings $W_1$ and $W_2$ of $M$, let  $\tilde{L_1}$ and $\tilde{L_2}$ be covering spaces over $L_1$ and $L_2$ such that $\partial\tilde{L_1}\cong\partial\tilde{L_2}$ as coverings over $\Lambda$. Then ${HW^{(0)}}^+(\tilde{L_0})\cong {HW^{(0)}}^+(\tilde{L_1})$.
    \end{prop}

\begin{proof} The proof of proposition \ref{adcleg} follows by proving that the assumptions imply  the differential in $HW^+(L_i)$ is defined using  moduli spaces that  do not enter the Liouville filling by a neck stretching argument. In this setting, let $u\in\mathcal{M}_0(x,y)$ be a rigid trajectory between Reeb chords which do not enter the filling. Then $\alpha_0(u)$ defines the same path in $[1,\infty]\times \Lambda$ whether it is seen in the completion of  $L_1$ or in the completion of $L_2$. Therefore after fixing a lift $\tilde{x}_0$ of $x(0)$, $\alpha_0(u)$ lifts to the same path in the cylindrical ends of $\tilde{L_1}$ and $\tilde{L_2}$ by uniqueness of lifting. Therefore the differentials coincide in ${HW^{(0)}}^+(\tilde{L_0})$ and ${HW^{(0)}}^+(\tilde{L_1})$ implying ${HW^{(0)}}^+(\tilde{L_0})\cong {HW^{(0)}}^+(\tilde{L_1})$.
    \end{proof}

In the following theorem, we follow an idea of Zhou (\cite{zhou2021symplectic},Theorem 3.17) where he managed to make the boundaries of the  covering spaces match in order to use invariance of positive symplectic homology and its structure map by matching .
\begin{thm} \label{funda1}
        Let $W^{2n}$ be a Liouville filling of $(S^{2n-1},\xi_{std})$ with $n>2$  and let $L$ be an exact Lagrangian filling of the standard Legendrian $\Lambda_{std}$ with vanishing Maslov class. Then $\pi_1(L)=1$.
\end{thm}
\begin{proof}
    The Lagrangian $L$ is proven to be spin in \cite{kim2024topology},section 2.3.1. Therefore we can work with wrapped Floer theory with $\mathbb{Z}$ coefficients. By \cite{chantrainefloer} we know that for an exact Maslov zero Lagrangian filling $L_1$ of $\Lambda_{std}$ in the standard ball $D^{2n}$,  $\pi_1(L_1)=1$. Therefore the positive wrapped Floer homology group of the universal cover $${HW^{(0)}}^+(\tilde{L_1})\cong HW^+(L_1).$$  Since $HW(D^{n}_{std})=0$, the long exact sequence \ref{les wcs} in wrapped Floer homology implies $${HW_{n-1}^{(0)}}^+(\tilde{L_1})\cong HW_{n-1}^+(L_1)\cong H^0(L_1)\cong\mathbb{Z}$$  
    
    Now considering the Liouville filling $W$, note that the map $\pi_1(\Lambda)\rightarrow\pi_1(L)$ induced by inclusion is injective since $\pi_1(\Lambda_{std})=0$. Therefore the the boundary $\partial\tilde{L}$ of the universal cover $\tilde{L}$ is given by $\pi_1(L)/\pi_1(\Lambda)=\pi_1(L)$ copies of $\Lambda_{std}$. Therefore  $\mathcal{L}:=L_1\times \pi_1(L)$ is a cover over $L_1\subset D^{2n}_{std}$ such that $\partial \mathcal{L}=\Lambda_{std}\times\pi_1(L)$, coinciding with $\partial \tilde{L}$.  
    
    The Legendrian $\Lambda_{std}$ and $(S^{2n-1},\xi_{std})$ are asymptotically dynamically convex and they satisfy the assumptions of proposition \ref{adccover}. Hence by proposition \ref{adccover} $$ {HW_{n-1}^{(0)}}^+(\tilde{L})\cong {HW_{n-1}^{(0)}}^+(\mathcal{L})=\Pi_{\pi_1(L)}HW_{n-1}(L_1)\cong\Pi_{\pi_1(L)}\mathbb{Z} $$
    The long exact sequence of the wrapped Floer homology of the universal cover $\tilde{L}$ (proposition \ref{les wcs}) implies $ H^0(\tilde{L})\cong {HW_{n-1}^{(0)}}^+(\tilde{L}) $. Since $H^0(\tilde{L})=\mathbb{Z}$, we have $\mathbb{Z}\cong\Pi_{\pi_1(L)}\mathbb{Z}.$ This implies $\pi_1(L)=1$.
\end{proof}

\begin{cor}[Theorem \ref{lagrangian filling}]\label{lagrangian filling proof}
    Let $W^{2n}$ be a Liouville filling of $(S^{2n-1},\xi_{std})$  and let $L$ be an exact Lagrangian filling of the standard Legendrian $\Lambda_{std}$ with vanishing Maslov class. Then $L$ is diffeomorphic to the ball $B^n$ for $n\neq 1,2,4$ and homeomorphic to $B^n$ when $n=4$ .
\end{cor}
\begin{proof}
    For a Maslov zero exact Lagrangian filling, $\pi_1(L)=0$ by theorem \ref{funda1} for $n>2$. For Maslov-zero exact fillings, $H_i(L,\partial L)\cong H_i(B_n,\partial B_n)$ by theorem A in \cite{kim2024topology}. Therefore by h-cobordism theorem, $L$ is diffeomorphic to the ball $B^n$ when $n\geq 6$. . The $n=3$ case is implied by Perelman's proof of the Poincare conjecture. When $n=4$, $L$ is homeomorphic to $B^4$ by Friedman's theorem\cite{freedman1982topology}. The proof of smooth Poincare conjecture for smooth homotopy $5-$spheres (\cite{kervaire1963groups}) implies the $n=5$ case. The $n=2$ case follows after noticing that that $(S^3,\xi_{std})$ has a unique Liouville filling by the standard ball and then applying the Eliashberg-Polterovich theorem (\cite{eliashberg1996local}) on uniqueness of exact Lagrangian filling.
    \end{proof}
\begin{rem}
   Let $W$ be a topologically simple Liouville filling of a contact manifold $M$ that have vanishing symplectic homology and nice enough Reeb dynamics (for instance lacunary) on $M$. We can infer $\pi_1(W)=0$ using symplectic homology of covering spaces in a similar fashion to the proofs in the  Lagrangian setting of theorem \ref{funda}. 
\end{rem}

\subsection{$\pi_i$ of symplectic fillings} \label{4.4}

 Given the long exact sequence in symplectic homology from proposition (\ref{les}) $...\rightarrow SH_*(W)\rightarrow SH^+_*(W)\rightarrow H^{n-*+1}(W)\rightarrow..$, we can look at the map $$\delta_\partial:SH^+_*(W)\rightarrow H^{n-*+1}(W)\rightarrow H^{n-*+1}(\partial W)$$ where $H^{n-*+1}(W)\rightarrow H^{n-*+1}(\partial W)$ is the map induced by inclusion. Zhou(\cite{zhou2021symplectic} Theorem A) proves that for  $\partial W$ that are ADC, the map $\delta_\partial:SH^+_*(W)\rightarrow H^{n-*+1}(\partial W)$ is independent of topologically simple exact fillings.

If we work with symplectic cohomology $SH^*(W)$ instead of symplectic homology $SH_*(W)$, then for ADC $\partial W$ we have the dual map $\partial_\delta:H_{*+n+1}(\partial W,k)\rightarrow H_{*+n+1}(W)\rightarrow SH^{-*}_+(\partial W,k)$ is independent of topologically simple Liouville fillings. This statement is implied by Zhou's proof of \cite{zhou2021symplectic} Theorem A since  we can use the same moduli-spaces involved in the proof but define the differential by taking outgoing trajectories instead of incoming ones. Similarly we present a dual version of a proposition by Zhou whose proof is also implicit.

\begin{prop}[\cite{zhou2021symplectic},Proposition 3.16]\label{zhouprop} Let $M$ be an ADC-contact manifold and $W$ be a topologically simple exact filling of $M$ . Then for the universal covering $\tilde{W}\rightarrow W$ the map $\tilde{\partial_\delta}:H_{*+n+1}(\partial \tilde{W},k)\rightarrow SH^{-*}_+(\partial {W},k[\pi_1(W)])$ is independent of topologically simple exact filling $W'$ and covering $\tilde{W'}$ as long as $\partial\tilde{W}=\partial\tilde{W'}$.
    
\end{prop}

\begin{rem}
    Note that Zhou uses grading by $-\mu_{CZ}$ and therefore follows a cohomological convention for symplectic homology.Symplectic cohomology in \cite{zhou2021symplectic} is equivalent to the construction of symplectic homology in this paper.
\end{rem}

\begin{rem}
    Note that Zhou (\cite{zhou2021symplectic}) uses a quadratic Hamiltonian to define symplectic homology. Therefore to have a well defined differential in symplectic cohomology using a quadratic Hamiltonian, we must replace the direct sum in  $SC(W,H,J):=\bigoplus _{\gamma\in \mathcal{P}(H)}{\mathbb{Z}}\gamma$ by a direct product $SC(W,H,J):=\Pi _{\gamma\in \mathcal{P}(H)}{\mathbb{Z}}\gamma$. We could also stick to a direct sum construction and replace $SH^k$ by an action bounded version (see remark \ref{action bound}). One could also try to adapt Zhou's proof to linear Hamiltonians.
\end{rem}
\begin{thm}[\label{zh}Theorem\ref{zhou}]
    Let $M$ be an ADC contact manifold with a topologically simple exact filling $W$ such that $SH_*(W)=0$ and $\pi_i(M)\rightarrow \pi_i(W)$ is an isomorphism for $i=1,2$. Then for any other topologically simple exact filling $W'$, $\pi_i(M)\rightarrow \pi_i(W')$ is an isomorphism for $i=1,2$.
\end{thm}
\begin{proof}
    For $i=1$, $\pi_1(M)\rightarrow \pi_1(W')$ is an isomorphism by \cite{zhou2021symplectic}(Theorem E). Therefore for the universal cover $\tilde{W'}$ of $W'$ and universal cover $\tilde{M}$ of $M$ we have $\partial \tilde{W'}=\tilde{M}$. Therefore by proposition \ref{zhouprop}, the map $\tilde{\partial_\delta}:\partial_\delta:H_{*+n+1}(\partial \tilde{W'},k)\rightarrow SH^{-*}_+(\partial {W'},k[\pi_1(W)])$ is independent for topologically simple exact filling $W'$. 
   
    Since symplectic homology vanishes, so does symplectic cohomology $SH^*(M,k)$ by proposition \ref{adual} and so does symplectic cohomology of the universal cover $SH^{-*}_+(\partial {W},k[\pi_1(W)])$ by theorem \ref{vansc}. This implies by proposition \ref{lesloc2},that $H_{*+n+1}(\tilde{W'})\cong SH^{-*}_+(\partial W',k[\pi_1(W')])$. We can make the identification $\pi_2(M)\cong H_2(\tilde{M})$ and $\pi_2(W)\cong H_2(\tilde{W})$ since the second homology of universal cover is isomorphic to $\pi_2$ by Hurewicz theorem and since $\pi_2(\tilde{X})\cong \pi_2(X)$ for any space $X $ and its covering $\tilde{X}$. Therefore $H_2(\tilde{M},k)\cong H_2(\tilde{W},k)$ by universal coefficients theorem applied to $H_2(\tilde{M},
    \mathbb{Z})\cong H_2(\tilde{W},\mathbb{Z})$. By proposition \ref{zhouprop}, $H_2(\tilde{M},k)\cong H_2(\tilde{W'},k)$ for all topologically simple Liouville fillings . Universal coefficients theorem and the structure theorem for finitely generated abelian groups implies $H_2(\tilde{M})\cong H_2(\tilde{W})$. The previous isomorphism  is a  verification by first comparing the ranks of the two group followed by iteratively checking $\mathbb{Z}_{p^m}$ part in their decomposition using the properties $Tor(\mathbb{Z}_n,\mathbb{Z}_m)\cong \mathbb{Z}_m\otimes\mathbb{Z}_n\cong\mathbb{Z}_{gcd(m,n)}$.After identifying these groups with $\pi_2$, we have $\pi_2(M)\cong\pi_2(W')$.
\end{proof}

Now for higher $\pi_i$, we prove theorem \ref{higherhom}.
\begin{thm}[Theorem \ref{higherhom}]\label{higherhomproof}
    Let $M$ be an ADC contact manifold with a topologically simple exact filling $W$ such that $SH_*(W)=0$ and $\pi_i(M)\rightarrow \pi_i(W)$ is an isomorphism for $i\leq k$. Then for any other topologically simple exact filling $W'$, $\pi_i(M)\rightarrow \pi_i(W')$ is an isomorphism for $i\leq k-2$.
\end{thm}
\begin{proof}
    The proof of cases till $k=4$ is implied by theorem \ref{zhou}. The hypothesis $\pi_i(M)\rightarrow \pi_i(W)$ for $i\leq k$ implies $\pi_i(W,M)=0$ for $i\leq k$. Using relative Hurewicz theorem on the universal covers $\tilde{W}$ and $\tilde{M}$ so that one doesn't have to factor out by an action of $\pi_1(M)$ , we have $H_i(\tilde{W},\tilde{M})=0$ for $i\leq k$. Now the long exact sequence of homology groups of the pair $(\tilde{W},\tilde{M})$ implies $H_i(\tilde{M})\rightarrow H_i(\tilde{W})$ is an isomorphism for $i\leq k-1$
  
    As in the proof of theorem \ref{zhou}, using proposition \ref{zhouprop} we get $H_i(\tilde{M})\cong H_i(\tilde{W'})$ for $i\leq k-1$. By the long exact sequence of relative homology groups, this implies $H_i(\tilde{W'},\tilde{M})=0$ for $i\leq k-1$
    and thereby  $\pi_i(\tilde{W'},\tilde{M})=0$ for $i\leq k-1$ by relative Hurewicz theorem. The long exact sequence of homotopy groups now implies $\pi_i(\tilde{M})\cong \pi_i(\tilde{W'})$ for $i\leq k-2$. Since $\pi_i(\tilde{X})\cong\pi_i(X)$ for $i>1$, for any space $X$ and its universal cover $\tilde{X}$. We can conclude $\pi_i({M})\cong \pi_i({W'})$ for $i\leq k-2$.
    \end{proof}

\bibliographystyle{alpha}
\bibliography{ref_conv}

@article{dardennes2024symplectic,
  title={Symplectic non-convexity of toric domains},
  author={Dardennes, Julien and Gutt, Jean and Zhang, Jun},
  journal={Communications in Contemporary Mathematics},
  volume={26},
  number={04},
  pages={2350010},
  year={2024},
  publisher={World Scientific}
}

@article{zhou2021symplectic,
  title={Symplectic fillings of asymptotically dynamically convex manifolds I},
  author={Zhou, Zhengyi},
  journal={Journal of Topology},
  volume={14},
  number={1},
  pages={112--182},
  year={2021},
  publisher={Wiley Online Library}
}

@article{zhou2023fillings,
  title={On fillings of $\partial(V\times$ D)},
  author={Zhou, Zhengyi},
  journal={Mathematische Annalen},
  volume={385},
  number={3},
  pages={1493--1520},
  year={2023},
  publisher={Springer}
}

@article{bourgeois2009exact,
  title={An exact sequence for contact-and symplectic homology},
  author={Bourgeois, Fr{\'e}d{\'e}ric and Oancea, Alexandru},
  journal={Inventiones mathematicae},
  volume={175},
  number={3},
  pages={611--680},
  year={2009},
  publisher={Springer}
}

@article{bowden2022tight,
  title={Tight contact structures without symplectic fillings are everywhere},
  author={Bowden, Jonathan and Gironella, Fabio and Moreno, Agustin and Zhou, Zhengyi},
  journal={arXiv preprint arXiv:2211.03680},
  year={2022}
}

@article{cieliebak2018symplectic,
  title={Symplectic homology and the Eilenberg--Steenrod axioms},
  author={Cieliebak, Kai and Oancea, Alexandru},
  journal={Algebraic \& Geometric Topology},
  volume={18},
  number={4},
  pages={1953--2130},
  year={2018},
  publisher={Mathematical Sciences Publishers}
}

@article{mclean2016reeb,
  title={Reeb orbits and the minimal discrepancy of an isolated singularity},
  author={McLean, Mark},
  journal={Inventiones mathematicae},
  volume={204},
  number={2},
  pages={505--594},
  year={2016},
  publisher={Springer}
}

@article{chaidez20223d,
  title={3D convex contact forms and the Ruelle invariant},
  author={Chaidez, Julian and Edtmair, Oliver},
  journal={Inventiones mathematicae},
  volume={229},
  number={1},
  pages={243--301},
  year={2022},
  publisher={Springer}
}

@article{lazarev2020contact,
  title={Contact manifolds with flexible fillings},
  author={Lazarev, Oleg},
  journal={Geometric and Functional Analysis},
  volume={30},
  number={1},
  pages={188--254},
  year={2020},
  publisher={Springer}
}

@article{hofer1999characterization,
  title={A characterization of the tight 3-sphere II},
  author={Hofer, Helmut and Wysocki, Kris and Zehnder, K},
  journal={Communications on Pure and Applied Mathematics: A Journal Issued by the Courant Institute of Mathematical Sciences},
  volume={52},
  number={9},
  pages={1139--1177},
  year={1999},
  publisher={Wiley Online Library}
}

@article{kwon2024dynamically,
  title={On dynamically convex contact manifolds and filtered symplectic homology},
  author={Kwon, Myeonggi and Oba, Takahiro},
  journal={Journal of the London Mathematical Society},
  volume={109},
  number={5},
  pages={e12914},
  year={2024},
  publisher={Wiley Online Library}
}

@article{abreu2017multiplicity,
  title={Multiplicity of periodic orbits for dynamically convex contact forms},
  author={Abreu, Miguel and Macarini, Leonardo},
  journal={Journal of Fixed Point Theory and Applications},
  volume={19},
  number={1},
  pages={175--204},
  year={2017},
  publisher={Springer}
}

@article{chantrainefloer,
  title={Floer homology and Lagrangian concordance},
  author={Chantraine, Baptiste and Rizell, Georgios Dimitroglou and Ghiggini, Paolo and Golovko, Roman}
}

@article{viterbo1999functors,
  title={Functors and computations in Floer homology with applications, I},
  author={Viterbo, Claude},
  journal={Geometric \& Functional Analysis GAFA},
  volume={9},
  number={5},
  pages={985--1033},
  year={1999},
  publisher={Springer}
}

@article{ginzburg2021dynamical,
  title={Dynamical convexity and closed orbits on symmetric spheres},
  author={Ginzburg, Victor L and Macarini, Leonardo},
  journal={Duke Mathematical Journal},
  volume={170},
  number={6},
  pages={1201--1250},
  year={2021},
  publisher={Duke University Press}
}

@article{ritter2013topological,
  title={Topological quantum field theory structure on symplectic cohomology},
  author={Ritter, Alexander F},
  journal={Journal of Topology},
  volume={6},
  number={2},
  pages={391--489},
  year={2013},
  publisher={Wiley Online Library}
}

@article{damian2012floer,
  title={Floer homology on the universal cover, Audin’s conjecture and other constraints on Lagrangian submanifolds},
  author={Damian, Mihai},
  journal={Commentarii Mathematici Helvetici},
  volume={87},
  number={2},
  pages={433--463},
  year={2012}
}

@article{abouzaid2010open,
  title={An open string analogue of Viterbo functoriality},
  author={Abouzaid, Mohammed and Seidel, Paul},
  journal={Geometry \& Topology},
  volume={14},
  number={2},
  pages={627--718},
  year={2010},
  publisher={Mathematical Sciences Publishers}
}

@article{kim2024topology,
  title={On the topology of Lagrangian fillings of the standard Legendrian sphere},
  author={Kim, Joontae and Kwon, Myeonggi},
  journal={Journal of Symplectic Geometry},
  volume={22},
  number={3},
  pages={525--547},
  year={2024}
}

@inproceedings{cieliebak2010rabinowitz,
  title={Rabinowitz Floer homology and symplectic homology},
  author={Cieliebak, Kai and Frauenfelder, Urs and Oancea, Alexandru},
  booktitle={Annales scientifiques de l'Ecole normale sup{\'e}rieure},
  volume={43},
  number={6},
  pages={957--1015},
  year={2010}
}

@article{auroux2013beginner,
  title={A beginner's introduction to fukaya categories},
  author={Auroux, Denis},
  journal={arXiv preprint arXiv:1301.7056},
  year={2013}
}

@article{damian2009constraints,
  title={Constraints on exact Lagrangians in cotangent bundles of manifolds fibered over the circle},
  author={Damian, Mihai},
  journal={Commentarii Mathematici Helvetici},
  volume={84},
  number={4},
  pages={705--746},
  year={2009}
}

@article{Zhou_2024,
   title={On fillings of contact links of quotient singularities},
   volume={17},
   ISSN={1753-8424},
   url={http://dx.doi.org/10.1112/topo.12329},
   DOI={10.1112/topo.12329},
   number={1},
   journal={Journal of Topology},
   publisher={Wiley},
   author={Zhou, Zhengyi},
   year={2024},
   month=Mar }

@article{oancea2012topology,
  title={On the topology of fillings of contact manifolds and applications},
  author={Oancea, Alexandru and Viterbo, Claude},
  journal={Commentarii mathematici helvetici},
  volume={87},
  number={1},
  pages={41--69},
  year={2012}
}

@article{albers2017local,
  title={Local systems on the free loop space and finiteness of the Hofer-Zehnder capacity},
  author={Albers, Peter and Frauenfelder, Urs and Oancea, Alexandru},
  journal={Mathematische Annalen},
  volume={367},
  number={3},
  pages={1403--1428},
  year={2017},
  publisher={Springer}
}

@article{barraud2007lagrangian,
  title={Lagrangian intersections and the Serre spectral sequence},
  author={Barraud, Jean-Fran{\c{c}}ois and Cornea, Octav},
  journal={Annals of mathematics},
  pages={657--722},
  year={2007},
  publisher={JSTOR}
}

@article{barraud2024floer,
  title={Floer Homology with DG Coefficients. Applications to cotangent bundles},
  author={Barraud, Jean-Fran{\c{c}}ois and Damian, Mihai and Humiliere, Vincent and Oancea, Alexandru},
  journal={arXiv preprint arXiv:2404.07953},
  year={2024}
}

@article{kervaire1963groups,
  title={Groups of homotopy spheres: I},
  author={Kervaire, Michel A and Milnor, John W},
  journal={Annals of mathematics},
  volume={77},
  number={3},
  pages={504--537},
  year={1963},
  publisher={JSTOR}
}

@article{freedman1982topology,
  title={The topology of four-dimensional manifolds},
  author={Freedman, Michael Hartley},
  journal={Journal of Differential Geometry},
  volume={17},
  number={3},
  pages={357--453},
  year={1982},
  publisher={Lehigh University}
}

@misc{shahul2026topologicalobstructionsdynamicalconvexity,
      title={Topological Obstructions to Dynamical Convexity}, 
      author={Shahnaz Shamim Shahul},
      year={2026},
      eprint={2512.03893},
      archivePrefix={arXiv},
      primaryClass={math.SG},
      url={https://arxiv.org/abs/2512.03893}, 
}

@article{mcduff1991symplectic,
  title={Symplectic manifolds with contact type boundaries},
  author={McDuff, Dusa},
  journal={Inventiones mathematicae},
  volume={103},
  number={1},
  pages={651--671},
  year={1991},
  publisher={Springer}
}

@article{eliashberg1996local,
  title={Local Lagrangian 2-knots are trivial},
  author={Eliashberg, Yakov and Polterovich, Leonid},
  journal={Annals of mathematics},
  volume={144},
  number={1},
  pages={61--76},
  year={1996},
  publisher={JSTOR}}

@article{zhoudulagrangian,
  title={On uniqueness of the Lagrangian fillings},
  author={Du, Lite and Zhou, Zhengyi},
  journal={in preparation},
  
  year={},
  publisher={}}

\end{document}